\documentclass{article}
\usepackage{amsmath,amsthm,amssymb}
\usepackage{mathrsfs}

\newtheorem{theo}{Theorem}[section]
\newtheorem{prop}[theo]{Proposition}
\newtheorem{coro}[theo]{Corollary}
\newtheorem{lemm}[theo]{Lemma}

\theoremstyle{definition}
\newtheorem{defi}[theo]{Definition}
\newtheorem*{rema*}{Remark}
\newtheorem*{claim*}{Claim}
\newtheorem{assump}{Assumption}
\newtheorem{notation}{Notation}

\author{Shota FUKUSHIMA\thanks{Graduate School of Mathematical Sciences, the University of Tokyo, 3-8-1 Komaba, Meguro-ku, Tokyo, 153-8914, Japan. 
Email: fukusima@ms.u-tokyo.ac.jp} \thanks{The author is supported by Leading Graduate Course for Frontiers of Mathematical Sciences and Physics (FMSP), at Graduate School of Mathematical Science, the University of Tokyo. }}
\title{$L^2$ boundedness of pseudodifferential operators on manifolds with ends}

\newcommand{\pprime}{{\prime\prime}}
\newcommand{\transp}[1]{{}^t\!{#1}}

\newcommand{\jbracket}[1]{\left\langle {#1} \right\rangle}
\newcommand{\Op}{\mathrm{Op}}

\begin{document}
\maketitle
\begin{abstract}
We investigate properties of pseudodifferential operators on $L^2$ space on manifold with ends including asymptotically conical or hyperbolic ends. Our pseudodifferential operators are a generalization of the canonical quantization which naturally appears in the quantum mechanics on curved spaces. We prove a Calder\'on-Vaillancourt type theorem for our pseudodifferential operators and discuss a construction of parametrix of elliptic differential operators on manifolds with ends. 
\end{abstract}

\section{Introduction}\label{def_manifold_with_ends}
We consider a class of non-compact manifolds with suitable properties: 
\begin{defi}
An $n$-dimensional connected oriented manifold $\mathcal{M}$ is called a \textit{manifold with ends} if there exists a compact subset $\mathcal{K}\subset \mathcal{M}$ such that $\mathcal{M}\setminus \mathrm{Int}\,\mathcal{K}$ is diffeomorphic as manifolds with boundaries to $[1, \infty)\times S$ for some $(n-1)$-dimensional oriented compact manifold $S$. Here $\mathrm{Int}\, \mathcal{K}$ is the interior of $\mathcal{K}$. 
$\mathcal{M}\setminus\mathcal{K}$ are called the \textit{ends} of $\mathcal{M}$. 
\end{defi}

If $\mathcal{M}\setminus \mathcal{K}$ is diffeomorphic to $(1, \infty)\times S$, we can consider the ``\textit{polar coordinates}'' $(r, \theta)\in (1, \infty)\times S$ using the local coordinates $\theta=(\theta^1, \ldots, \theta^{n-1})$ on $S$. 
This is defined as the image of local coordinates introduced below. We take an atlas on $\mathcal{M}$ which consists of finite local coordinates $\{\Psi_\iota: \mathcal{U}_\iota \to \mathcal{V}_\iota\}_{\iota\in I}$ in following way.  

\begin{enumerate}
\item Let $\mathcal{K}$ be a compact subset of $\mathcal{M}$ such that $\mathcal{M}\setminus\mathcal{K}$ forms the ends. We take a finite atlas $\{\Psi_\iota: \mathcal{U}_\iota\to \mathcal{V}_\iota\}_{\iota\in I_\mathcal{K}}$ on $\mathcal{K}$ where 
$\mathcal{U}_\iota$ and $\mathcal{V}_\iota$ are relatively compact open subsets of $\mathcal{M}$ and $\mathbb{R}^n$ respectively. 
\item Recall a diffeomorphism $\mathcal{M}\setminus\mathcal{K}\approx (1, \infty)\times S$. Since $S$ is a compact manifold, we can take a finite atlas $\{\Psi_\iota^\prime: \mathcal{U}_\iota^\prime\to \mathcal{V}_\iota^\prime\}_{\iota\in I_\infty}$. $\mathcal{U}_\iota^\prime$ and $\mathcal{V}_\iota^\prime$ are relatively compact open subsets of $S$ and $\mathbb{R}^{n-1}$ respectively. We set $\mathcal{U}_\iota:=(1, \infty)\times\mathcal{U}_\iota^\prime$, $\mathcal{V}_\iota:=(1, \infty)\times\mathcal{V}_\iota^\prime$ and $\Psi_\iota:=\mathrm{id}\times \Psi_\iota^\prime: \mathcal{U}_\iota\to \mathcal{V}_\iota$. 
\item Finally we put $I:=I_\mathcal{K}\cup I_\infty$. Obviously $I$ is a finite set. 
\end{enumerate}

We equip $\mathcal{M}$ with a metric $g$ with some boundedness conditions. 

\begin{assump}\label{assump_metric_M}
\begin{enumerate}
\item There exists a positive smooth function $\lambda: \mathcal{M}\to (0, \infty)$ such that, for each $\iota \in I_\infty$, $\lambda_\iota:=(\Psi_\iota^{-1})^*\lambda$ is a function dependent only on $r\in (1, \infty)$ and satisfies
\[
\partial_r^{j+1}\log \lambda_{\iota} \in L^\infty((1, \infty)) \quad \text{for all } j\geq 0. 
\] 
We also assume that $\inf_{x\in \mathcal{M}} \lambda(x)>0$. 
\item For each $\iota\in I_\infty$, the pullback of $g$ by $\Psi_\iota^{-1}: \mathcal{V}_\iota\to\mathcal{U}_\iota$ is the form
\begin{align*}
&(\Psi_\iota^{-1})^*g= \\
&G^\iota_{00}(r, \theta)dr^2+2\lambda_\iota(r)\sum_{j=1}^{n-1}G^\iota_{0j}(r, \theta)drd\theta^j+\lambda_\iota(r)^2\sum_{j, k=1}^{n-1}G^\iota_{jk}(r, \theta)d\theta^j d\theta^k. 
\end{align*}
$(G^\iota_{\mu\nu}(r, \theta))_{\mu, \nu=0}^{n-1}$ is a positive definite symmetric matrix with 
\begin{itemize}
\item $C^{-1}|v|^2\leq \sum_{\mu, \nu=0}^{n-1}G^\iota_{\mu\nu}(r, \theta)v^\mu v^\nu \leq C|v|^2$, for some $C>0$ independent of $(r, \theta)\in \mathcal{V}_\iota$ and $v\in \mathbb{R}^n$, 
\item For all $\mu$, $\nu\in \{0, \ldots, n-1\}$, integers $\alpha_0\in \mathbb{Z}_{\geq 0}$ and multiindices $\alpha\in \mathbb{Z}_{\geq 0}^{n-1}$, 
\[
(\lambda_\iota(r)^{-1}\partial_\theta)^\alpha \partial_r^{\alpha_0} G^\iota_{\mu\nu}(r, \theta)\in L^\infty(\mathcal{V}_\iota). 
\]
\end{itemize}
\end{enumerate}
\end{assump}

In following, we fix a Riemannian metric $g$ on $\mathcal{M}$ which satisfies Assumption \ref{assump_metric_M} and take $\lambda: \mathcal{M}\to (0, \infty)$ appearing there. 
This assumption includes not only the case of manifolds with conical ends $\lambda=r$ (for instance, the Euclidean space), but also that of manifolds with asymptotically hyperbolic ends $\lambda=\sinh r$ (for instance, the equivalent setting in Mazzeo-Melrose \cite{Mazzeo-Melrose87}). The above condition implies the global boundedness of the Riemann curvature tensor associated with $g$ and its covariant derivatives. 

Our main result is that we can construct parametrices of elliptic differential operators on $\mathcal{M}$. Let $\mathcal{P}: C^\infty(\mathcal{M})\to \mathcal{M})$ be a differential operator on $\mathcal{M}$. We assume that $\mathcal{P}$ is of the form 
\begin{equation}
(\Psi_\iota^{-1})^*\mathcal{P}\Psi_\iota^*=\sum_{\gamma_0+|\gamma|\leq m}a^\iota_{\gamma_0\gamma}(r, \theta)(i^{-1}\lambda_\iota^{-1}\partial_\theta)^\gamma (i^{-1}\partial_r)^{\gamma_0} \label{eq_P_local}
\end{equation}
on the coordinate neighborhood $\mathcal{U}_\iota$ ($\iota\in I_\infty$). $a^\iota_{\gamma_0\gamma}\in C^\infty(\mathcal{V}_\iota)$ are smooth functions such that, for all $(\alpha_0, \alpha)\in \mathbb{Z}_{\geq 0}\times \mathbb{Z}_{\geq 0}^{n-1}$, 
\[
(\lambda_\iota^{-1}\partial_\theta)^\alpha \partial_r^{\alpha_0}a^\iota_{\gamma_0\gamma}\in L^\infty(\mathcal{V}_\iota). 
\]
We denote the set consisting of such $\mathcal{P}$ by $\mathrm{Diff}^m(\mathcal{M}, \lambda)$. The principal symbol $\sigma(\mathcal{P}): T^*\mathcal{M}\to \mathbb{C}$ is 
\[
\sigma(\mathcal{P})(\tilde\Psi_\iota^{-1}(r, \theta, \rho, \eta))=\sum_{\gamma_0+|\gamma|= m}a^\iota_{\gamma_0\gamma}(r, \theta)(\lambda_\iota^{-1}\eta)^\gamma \rho^{\gamma_0} 
\]
if $(\Psi_\iota^{-1})^*\mathcal{P}\Psi_\iota^*$ is given by (\ref{eq_P_local}). Here $\tilde\Psi_\iota: T^*\mathcal{U}_\iota\to \mathcal{V}_\iota\times \mathbb{R}\times \mathbb{R}^{n-1}$ is the canonical coordinates associated with $\Psi_\iota: \mathcal{U}_\iota\to \mathcal{V}_\iota$. 
We call $\mathcal{P}\in \mathrm{Diff}^m(\mathcal{M}, \lambda)$ elliptic if, for all $z\in\mathbb{C}$ with $\mathrm{dist}(z, \sigma(\mathcal{P})(T^*\mathcal{M}))>0$, there exists a constant $C>0$ such that the inequality
\[
C^{-1}(1+|\xi|_{g^{-1}})^m\leq |z-\sigma(\mathcal{P})(x, \xi)| \leq C(1+|\xi|_{g^{-1}})^m
\]
holds for all $(x, \xi)\in T^*\mathcal{M}$. Here $g^{-1}$ is defined as $|\alpha|_{g^{-1}}^2:=\sum_{\mu, \nu} g^{\mu\nu}(x)\alpha_\mu \alpha_\nu$ if $\alpha=\sum_\mu \alpha_\mu dx^\mu\in T^*_x\mathcal{M}$, where $(g^{\mu\nu})$ is the inverse matrix of the metric tensor $(g_{\mu\nu})$. 

In the statement below, $L^2(\mathcal{M}, g)$ is the space of square integrable functions on $\mathcal{M}$ with respect to the measure $\sqrt{\det (g_{\mu\nu})}dx^1\cdots dx^n$ naturally defined by the metric $g$. 

\begin{theo}\label{thm_parametrix}
Let $\mathcal{P}\in \mathrm{Diff}^m(\mathcal{M}, \lambda)$ be an elliptic differential operator. Then, for each integer $N\geq 0$, we can construct $L^2(\mathcal{M}, g)$ bounded operators $\mathcal{Q}_N(z)$ such that 
\[
(z-\mathcal{P})\mathcal{Q}_N(z)u=u+\mathcal{R}_{N+1}(z)u
\]
for all $u\in C_c^\infty(\mathcal{M})$ and $z\in\mathbb{C}$ with $\delta(z, \sigma(\mathcal{P})):=\mathrm{dist}(z, \sigma(\mathcal{P})(T^*\mathcal{M}))>0$. $\mathcal{R}_{N+1}(z)$ are bounded operators on $L^2(\mathcal{M}, g)$ and 
\begin{align*}
&\|\mathcal{R}_{N+1}(z)\|_{L^2(\mathcal{M}, g)\to L^2(\mathcal{M}, g)} \\
&\leq
C\delta(z, \sigma(\mathcal{P}))^{-\frac{N+1}{m}}
\sum_{l=0}^{N+M}\sup_{\substack{(x, \xi)\in \\ T^*\mathcal{M}}}\left(\frac{(1+|\xi|_{g^{-1}})^m}{|z-\sigma(\mathcal{P})(x, \xi)|}\right)^{l+1-\frac{N+1}{m}}
\end{align*}
for some $C=C_N>0$ and $M\in \mathbb{N}$ independent of $z$ ($M$ is also independent of $N$). 
$\mathcal{Q}_N(z)$ and $\mathcal{R}_{N+1}(z)$ map $C_c^\infty(\mathcal{M})$ to $C^\infty(\mathcal{M})$ continuously. 
\end{theo}

In order to construct parametrices, we introduce a suitable class of pseudodifferential operators in order to investigate such differential operators in polar coordinates. 

\begin{notation}
It is convenient to introduce shorthand notations 
\[
q=(r, \theta)\in \mathbb{R}^n=\mathbb{R}\times \mathbb{R}^{n-1}
\]
and for the dual variable $(\rho, \eta)\in \mathbb{R}\times\mathbb{R}^{n-1}$ 
\[
p=(\rho, \eta)=\rho\oplus\eta\in \mathbb{R}^n=\mathbb{R}\times\mathbb{R}^{n-1}. 
\]
We also introduce notations
\[
\jbracket{p}:=(1+|p|^2)^\frac{1}{2}
\]
and
\[
\jbracket{\rho\oplus\eta}:=(1+\rho^2+|\eta|^2)^\frac{1}{2}. 
\]

We always split a multiindex as $A=(\alpha_0, \alpha)\in \mathbb{Z}_{\geq 0}\times\mathbb{Z}_{\geq 0}^{n-1}$. The length $|A|$ is defined as
\[
|A|=\alpha_0+|\alpha|=\alpha_0+\alpha_1+\cdots +\alpha_{n-1}
\]
if $\alpha=(\alpha_1, \ldots. \alpha_{n-1})$. 
We denote
\[
\partial_q^A=\partial_r^{\alpha_0}\partial_\theta^\alpha, \, \partial_p^B=\partial_\rho^{\beta_0}\partial_\eta^\beta. 
\]
\end{notation}

First we define a symbol class. 
\begin{defi}
Let $m\in \mathbb{R}$, $\sigma\in [0, 1]$. We fix a positive smooth function $\lambda: \mathbb{R}\to (0, \infty)$ satisfying the property 
\begin{equation}
C^{-1}\leq \frac{\lambda (r)}{\lambda(r^\prime)}\leq C \label{assump_lambda}
\end{equation}
for some $C>0$ independent of $r$, $r^\prime\in \mathbb{R}$ such that $|r-r^\prime|\leq 1$. 
This assumption also permits $\lambda(r)$ with $\inf \lambda(r)=0$ or without smoothness, for instance, $\lambda(r)=e^{-|r|}$. 

We define symbol classes $S^m_\sigma(\lambda)$ by the following condition: 
a smooth function $a\in C^\infty(\mathbb{R}^{2n})$ is in $S^m_\sigma(\lambda)$ if, for all $M\geq 0$, there exists a constant $C_M>0$ such that 
\[
|(\lambda(r)^{-1}\partial_\theta)^\alpha (\lambda(r)\partial_\eta)^\beta \partial_r^{\alpha_0} \partial_\rho^{\beta_0}a(q, p)|\leq C_M\jbracket{\rho\oplus \lambda(r)^{-1}\eta}^{m-\sigma |B|}
\]
for all $(q, p)=(r, \theta, \rho, \eta)\in \mathbb{R}^{2n}$ and $A=(\alpha_0, \alpha)$, $B=(\beta_0, \beta)\in \mathbb{Z}_{\geq 0}^n$ with $|A|+|B|\leq M$. 

The seminorms $|a|_{\lambda: M, m, \sigma}$ are defined by 
\begin{align*}
&|a|_{\lambda: M, m, \sigma} \\
&:=
\sum_{|A|+|B|\leq M}\|\jbracket{\rho\oplus \lambda(r)^{-1}\eta}^{-m+\sigma |B|}(\lambda(r)^{-1}\partial_\theta)^\alpha (\lambda(r)\partial_\eta)^\beta \partial_r^{\alpha_0} \partial_\rho^{\beta_0}a\|_{L^\infty(\mathbb{R}^{2n})}. 
\end{align*}
\end{defi}

\begin{rema*}
Our symbol class $S^m_\sigma(\lambda)$ is represented as $S(\jbracket{\rho\oplus \lambda^{-1}\eta}^m, \tilde g_\sigma)$, where 
\[
\tilde g_\sigma:=dr^2+\lambda(r)^2\sum_{j=1}^{n-1} (d\theta^j)^2+\jbracket{\rho\oplus \lambda(r)^{-1}\eta}^{-2\sigma}\left(d\rho^2+\lambda(r)^{-2}\sum_{j=1}^{n-1} (d\eta_j)^2\right), 
\]
in H\"ormander's notation \cite{Hormander85}. 
Since we permit the exponential increase for $\lambda(r)$, the metric $\tilde g_\sigma$ is not necessarily \textit{temperate}, that is, roughly speaking, the coefficients of the metric tensor are at most polynomially increasing. 
\end{rema*}

Next we define pseudodifferential operators with symbols in $S^m_\sigma(\lambda)$. We fix a positive function $g: \mathbb{R}^n\to (0, \infty)$. A typical example is $g(q)=\det (g_{\mu\nu}(q))_{\mu, \nu}$. In this case, the volume density associated with the metric tensor $(g_{\mu\nu})$ is $g(q)^{1/2}dq$. 

\begin{defi}
A pseudodifferential operator $\Op^{g, t}(a)$ with symbol $a\in S^m_\sigma(\lambda)$ for some $t\in [0, 1]$, $m\in \mathbb{R}$ and $\sigma\in [0, 1]$ is 
\begin{align*}
&\Op^{g, t}(a)u(q) \\
&:=\frac{1}{(2\pi)^n}g(q)^{-\frac{1}{4}}\int_{\mathbb{R}^n}dp \int_{\mathbb{R}^n}dq^\prime \, a\left(tq+(1-t)q^\prime, p\right)e^{ip\cdot (q-q^\prime)}u(q^\prime)g(q^\prime)^\frac{1}{4}
\end{align*}
\end{defi}

\begin{rema*}
\begin{itemize}
\item Our pseudodifferential operators are not necessarily properly supported. We can find the microlocal analysis on asymptotically hyperbolic manifolds by properly supported pseudodifferential operators in Bouclet \cite{Bouclet11} for instance. He used them in order to prove the Strichartz estimates on manifolds with asymptotically hyperbolic ends \cite{Bouclet13}. 
\item In the model case $g(q)=\det (g_{\mu\nu}(q))_{\mu, \nu}$, it is natural to regard $\Op^{g, t}(a)$ as an operator acting on half-densities $u|\mathrm{vol}_g|^{1/2}$ (see Guillemin-Sternberg \cite{Guillemin-Sternberg13} for more details on half-densities) as 
\[
\Op^{g, t}(a)(u|\mathrm{vol}_g|^\frac{1}{2})=(\Op^{g, t}(a)u)|\mathrm{vol}_g|^\frac{1}{2}. 
\]
Here $\mathrm{vol}_g$ is a volume form associated with the metric $g$ and $|\mathrm{vol}_g|^{1/2}$ is the non-vanishing section of the vector bundle of half-densities on $\mathbb{R}^n$ defined by $\mathrm{vol}_g$. The integral kernel of $\Op^{g, t}(a)$ as an operator on half-densities is 
\[
\frac{1}{(2\pi)^n}\int_{\mathbb{R}^n} a(tq+(1-t)q^\prime, p)e^{ip\cdot (q-q^\prime)}\, dp |dq|^\frac{1}{2}|dq^\prime|^\frac{1}{2} 
\]
in the distributional sense. One can interpret ``$dp|dq|^{1/2}|dq^\prime|^{1/2}\approx dp dq$'' near diagonal. The right hand side $dpdq$ is the Liouville measure associated with the canonical symplectic structure on $T^*\mathbb{R}^n$. 
\item One can regard our pseudodifferential operators as a generalization of the canonical quantization on curved spaces \cite{DeWitt57}. Although $q_j$ and $\eta_j$ do not belong to $S^m_\sigma(\lambda)$ in general, we can calculate their quantization directly and obtain
\[
\Op^{g, t}(q_j)u(q)=q_ju(q), \quad \Op^{g, t}(p_j)u(q)=\frac{1}{i}\partial_{q_j}u(q)+\frac{1}{4i}(\log g(q))u(q) 
\]
for fundamental observables $q$ and $p$. 

As operators on half-densities, $\Op^{g, t}(p_j)$ are the Lie derivative operators on half-densities: 
\[
\Op^{g, t}(p_j)(u|\mathrm{vol}_g|^\frac{1}{2})=\frac{1}{i}\mathcal{L}_{\partial_{q_j}}(u|\mathrm{vol}_g|^\frac{1}{2}). 
\]
\end{itemize}
\end{rema*}

We prove and use the Calder\'on-Vaillancourt type $L^2$ boundedness theorem \cite{Calderon-Vaillancourt71} for the estimate of our pseudodifferential operators. The difficulty in treating our pseudodifferential operators is that it is unknown whether we can obtain the off-diagonal decay of the integral kernel of them in the case of $\lambda$ with exponential increasing. The simplest version of $L^2$ boundedness theorem is as follows. 

\begin{theo}[$L^2$ boundedness]\label{thm_L2_bdd_simplest}
Let $\lambda: \mathbb{R}\to (0, \infty)$ be a function satisfying (\ref{assump_lambda}). Then there exist  constants $C>0$ and $M\geq 0$ such that, for all $a\in S^0_0(\lambda)$, all $u\in C_c^\infty(\mathbb{R}^n)$ and all $t\in [0, 1]$, the inequality
\[
\|\Op^{g, t}(a)u\|_{L^2(g^{1/2}dq)}\leq C|a|_{\lambda: M, 0, 0}\|u\|_{L^2(g^{1/2}dq)}
\]
holds. 
\end{theo}

We prove $L^2$ boundedness theorem for bounded bisymbols in Theorem \ref{thm_L2_bdd_bi}. Theorem \ref{thm_L2_bdd_simplest} above is a special case of the $L^2$ boundedness theorem for bisymbols. The method of proof is a scaling. This enables us to avoid the argument of off-diagonal decay of integral kernels. The bisymbols are introduced in Section \ref{sec_bisymbol} and the proof of $L^2$ boundedness is in Section \ref{sec_L2bdd}. 
In Section \ref{sec_parametrix} we investigate a composition of differential operators and pseudodifferential operators in local coordinates. In the final section \ref{sec_ess_selfadj}, we prove the main theorem \ref{thm_parametrix} and  we investigate the essential self-adjointness of symmetric differential operators on manifolds with ends as an application. 

\begin{rema*}
Our setting is related to manifolds with bounded geometry. A complete Riemannian manifold has bounded geometry if and only if it has the positive injectivity radius and its Riemannian curvature tensor and its covariant derivative are globally bounded. Analysis on manifolds with bounded geometry is in Grosse-Schneider \cite{Grosse-Schneider13}, Shubin \cite{Shubin92}, for example. 

Geometric definition of pseudodifferential operators on manifolds is also investigated recently. We can find some of them in Derezi\'nski-Latosi\'nski-Siemssen \cite{Derezinski-Latosinski-Siemssen20}, Levy \cite{Levy}. 
\end{rema*}

\section{Bisymbols}\label{sec_bisymbol}

\begin{defi}
Let $m\in\mathbb{R}$, $\sigma\in[0, 1]$ and $t\in[0, 1]$. 
A smooth function $a: \mathbb{R}^{3n} \to \mathbb{C}$ is an element of $BS^m_\sigma(\lambda; t)$ if and only if, for all indices $A=(\alpha_0, \alpha), B=(\beta_0, \beta)$ and $A^\prime=(\alpha_0^\prime, \alpha^\prime)\in \mathbb{Z}_{\geq 0}\times\mathbb{Z}_{\geq0}^{n-1}$, there exists a constant $C_{ABA^\prime t}>0$ such that
\[
|\partial_{q, l t}^A\partial_{p, l t}^B\partial_{q^\prime, l t}^{A^\prime}a(q, p, q^\prime)|\leq C_{ABA^\prime t}\jbracket{\rho\oplus (\lambda^t_{jk})^{-1}\eta}^{m-\sigma|B|}
\]
for all $|r-j|\leq 1$, $|r^\prime-k|\leq 1$ and $l=(j, k) \in \mathbb{Z}^2$. Here we introduced shorthand notations
\[
\partial_{q, l t}^A:=\partial_r^{\alpha_0}((\lambda^t_{jk})^{-1}\partial_\theta)^\alpha, 
\partial_{p, l t}^B:=\partial_\rho^{\beta_0}(\lambda^t_{jk}\partial_\eta)^\beta, 
\partial_{q^\prime, l t}^{A^\prime}:=\partial_{r^\prime}^{\alpha_0^\prime}((\lambda^t_{jk})^{-1}\partial_{\theta^\prime})^{\alpha^\prime}, 
\]
and 
\[
\lambda^t_{jk}:=\lambda(tj+(1-t)k). 
\]

We introduce seminorms $\{|\cdot|_{m, \sigma, M, t}\}_{M=0}^\infty$ on $BS^m_\sigma(\lambda; t)$ by
\begin{align*}
&|a|_{\lambda: M, m, \sigma, t} \\
&:=\sum_{\substack{|A|+|B|+|A^\prime| \\ \leq M}} \sup_{\substack{l=(j, k) \\ \in\mathbb{Z}^2}} \sup_{\substack{q, p, q^\prime \\ |r-j|\leq 1 \\ |r^\prime-k|\leq 1}}
\jbracket{\rho\oplus (\lambda^t_{jk})^{-1}\eta}^{-m+\sigma|B|}|\partial_{q, l t}^A\partial_{p, l t}^B\partial_{q^\prime, l t}^{A^\prime}a(q,p, q^\prime)|. 
\end{align*}
If $\lambda$ is clear from the context, we omit writing the letter $\lambda$ from $|a|_{\lambda: M, m, \sigma, t}$ such as $|a|_{M, m, \sigma, t}$. 
We put $BS^m_\sigma(\lambda):=\bigcup_{t\in[0, 1]}BS^m_\sigma(\lambda; t)$. 
\end{defi}

In arguments on $L^2$ boundedness of pseudodifferential operators, it is enough to assume the boundedness of the increasing or decreasing rate. 

The following proposition says that this concept of bisymbols is a generalization of the setting in the introduction. 
\begin{prop}\label{prop_s_sub_bs}
Let $\lambda$ be a positive function on $\mathbb{R}$ satisfying the condition (\ref{assump_lambda}). For any integers $M\geq 0$ and real parameters $t\in [0, 1]$, there exists a constant $C_{M t}>0$ such that  
\begin{equation}
|a_t|_{\lambda: M, m, \sigma, t}\leq C_{Mt}|a|_{\lambda: M, m, \sigma} \label{bs<s}
\end{equation}
for all $a\in S^m_\sigma(\lambda)$ if we put $a_t(q, p, q^\prime):=a(tq+(1-t)q^\prime, p)$. 
\end{prop}

\begin{proof}
If $|A|+|B|+|A^\prime|\leq M$, $l=(j, k)\in\mathbb{Z}^2$, $|r-j|\leq 1$ and $|r^\prime-k|\leq 1$, then 
\begin{align}
&|\partial_{q, l t}^A\partial_{p, l t}^B\partial_{q^\prime, l t}^{A^\prime}a_t(q,p, q^\prime)| \nonumber \\ =
&t^{|\alpha|}(1-t)^{|\alpha^\prime|}\left(\frac{\lambda(tr+(1-t)r^\prime)}{\lambda^t_{jk}}\right)^{|\alpha|+|\alpha^\prime|-|\beta|} \\
&\times |\partial_r^{\alpha_0+\alpha_0^\prime}(\lambda^{-1}\partial_\theta)^{\alpha+\alpha^\prime}\partial_\rho^{\beta_0}(\lambda\partial_\eta)^\beta a(tq+(1-t)q^\prime, p)| \nonumber \\ \leq 
&t^{|\alpha|}(1-t)^{|\alpha^\prime|}\left(\frac{\lambda(tr+(1-t)r^\prime)}{\lambda^t_{jk}}\right)^{|\alpha|+|\alpha^\prime|-|\beta|}|a|_{\lambda: M, m, \sigma}
\jbracket{\rho\oplus \lambda^{-1}\eta}^{m-\sigma|B|} \nonumber \\ =
&t^{|\alpha|}(1-t)^{|\alpha^\prime|}\left(\frac{\lambda(tr+(1-t)r^\prime)}{\lambda^t_{jk}}\right)^{|\alpha|+|\alpha^\prime|-|\beta|} \label{bdd_bs} \\
&\times\left(\frac{\jbracket{\rho\oplus \lambda(tr+(1-t)r^\prime)^{-1}\eta}}{\jbracket{\rho\oplus (\lambda^t_{jk})^{-1}\eta}}\right)^{m-\sigma|B|} 
|a|_{\lambda: M, m, \sigma}
\jbracket{\rho\oplus (\lambda^t_{jk})^{-1}\eta}^{m-\sigma|B|}. \nonumber
\end{align}
The term (\ref{bdd_bs}) above has an upper bound independent of $l=(j, k)$ since the condition (\ref{assump_lambda}) for $\lambda$ and $|(tj+(1-t)k)-(tr+(1-t)r^\prime)|\leq 1$ if $|r-j|$ and $|r^\prime-k|\leq 1$. Hence 
\[
|a_t|_{\lambda: |A|+|B|+|A^\prime|, m, \sigma, t}\leq C_{ABA^\prime t}|a|_{\lambda: M, m, \sigma}. 
\]
This completes the proof. 
\end{proof}

We define pseudodifferential operators with bisymbols. 
\begin{defi}
We define $\Op^g(a)$ associated with a bisymbol $a(q, p, q^\prime)\in BS^m_\sigma(\lambda)$ by 
\[
\Op^g(a)u(q):=\frac{1}{(2\pi)^n}g(q)^{-\frac{1}{4}}\int_{\mathbb{R}^n}dp \int_{\mathbb{R}^n}dq^\prime \, a\left(q, p, q^\prime\right)e^{ip\cdot (q-q^\prime)}u(q^\prime)g(q^\prime)^\frac{1}{4}. 
\]
$\Op^1(a)$ is simply denoted as $\Op(a)$. Clearly we have $\Op^g(a)=g^{-1/4}\Op(a)g^{1/4}$. 
\end{defi}

We first prove the smoothness of $\mathrm{Op}(a)u$ for $u\in C_c^\infty(\mathbb{R}^n)$. 

\begin{prop}\label{prop_smoothness}
If $a\in BS^m_0(\lambda)$, then $\Op(a)$ defines a continuous linear operator from $C_c^\infty(\mathbb{R}^n)$ to $C^\infty(\mathbb{R}^n)$. Thus so does $\Op^g(a)$. 
\end{prop}

This is proved by the same method as that for usual pseudodifferential operators, but we give the proof, since we need some care for the compact support condition of the test functions. 
We prepare a lemma for the proof. 
\begin{lemm}\label{lemm_bs_comp}
Let $K$ and $K^\prime$ be compact subsets of $\mathbb{R}^n$ and $t\in[0, 1]$. Then for all $M\geq 0$, there exists a constant $C_{MKK^\prime t}$ such that 
\[
|\partial_q^A\partial_p^B\partial_{q^\prime}^{A^\prime} a(q, p, q^\prime)|\leq C_{MKK^\prime t} |a|_{\lambda: M, m, 0, t}\jbracket{p}^m
\]
for all $a\in BS^m_0(\lambda; t)$, $|A|+|B|+|A^\prime|\leq M$ and $(q, q^\prime)\in K\times K^\prime$. 
\end{lemm}

\begin{proof}
Take $(q, q^\prime)\in K\times K^\prime$ and choose $l=(j, k)\in \mathbb{Z}^2$ such that $|r-j|, |r^\prime-k|\leq 1$. Then for all $A=(\alpha_0, \alpha), B=(\beta_0, \beta), A^\prime=(\alpha_0^\prime, \alpha^\prime)\in \mathbb{Z}_{\geq 0}\times\mathbb{Z}_{\geq 0}^{n-1}$ such that $|A|+|B|+|A^\prime|\leq M$, 
\[
|\partial_q^A\partial_p^B\partial_{q^\prime}^{A^\prime} a(q, p, q^\prime)|\leq (\lambda^t_{jk})^{|\alpha|-|\beta|+|\alpha^\prime|}\jbracket{\rho\oplus (\lambda^t_{jk})^{-1}\eta}^m|a|_{\lambda: M, m, 0, t}. 
\]
Since $l$ is in a finite set 
\[
(\mathbb{Z}\times\mathbb{Z})\cap (\{ j\mid \mathrm{dist}(j, \pi_\mathrm{rad}(K))\leq 1\}\times\{ k\mid \mathrm{dist}(k, \pi_\mathrm{rad}(K^\prime))\leq 1\})
\]
where $\pi_\mathrm{rad}(r, \theta):=r$ is a projection to a radial component, we have $C_{KK^\prime t}^{-1}\leq \lambda^t_{jk} \leq C_{KK^\prime t}$ for all such $l$. Thus 
\[
(\lambda^t_{jk})^{|\alpha|-|\beta|+|\alpha^\prime|}\leq C_{MKK^\prime t}
, \quad
\sup_{\rho, \eta} \jbracket{\rho\oplus (\lambda^t_{jk})^{-1}\eta}/\jbracket{\rho\oplus\eta}\leq C_{KK^\prime t}
\]
and hence obtain the result. 
\end{proof}

\begin{proof}[Proof of Proposition \ref{prop_smoothness}]
The integrand of $\Op(a)u(q)=(2\pi)^{-n}\int F(q, p)\, dp $ is 
\[
F(q, p)=e^{iq\cdot p}\mathscr{F}_{q^\prime\to p}[a(q, p, q^\prime)u(q^\prime)](p). 
\]
if we denote the Fourier transform of $u\in C_c^\infty (\mathcal{V})$ by $\mathscr{F}u(p)=\int u(q)e^{-ip\cdot q}\, dq$. 
Note that the function $q^\prime\mapsto a(q, p, q^\prime)u(q^\prime)$ is smooth and compactly supported in $\mathbb{R}^n$ by $u\in C_c^\infty(\mathbb{R}^n)$. 

The derivative of $F$ is
\begin{align*}
&D_q^A (e^{ip\cdot q}\mathscr{F}_{q^\prime\to p}[a(q, p, q^\prime)u(q^\prime)](p)) \\
=&\sum_{B\leq A}
\begin{pmatrix}
A \\
B
\end{pmatrix}
e^{ip\cdot q}p^{A-B} 
\mathscr{F}_{q^\prime\to p}[t^{|B|}(D_q^B a)(q, p, q^\prime)u(q^\prime)](p),
\end{align*}
and, for an arbitrary integer $N$, this is equal to
\begin{align*}
&\sum_{B\leq A}
\begin{pmatrix}
A \\
B
\end{pmatrix}
t^{|B|}e^{ip\cdot q}p^{A-B} 
\mathscr{F}_{q^\prime\to p}[\transp{L}^N((D_q^B a)(q, p, q^\prime)u(q^\prime))](p)
\end{align*}
by the integration by parts by the differential operator $L=(1+ip\cdot \partial_{q^\prime})/\jbracket{p}^2$. 
Hence
\begin{align}
&|\partial_q^A F(q, p)| \nonumber \\
&\leq C_A\sum_{B \leq A} |(ip)^{A-B} \mathscr{F}_{q^\prime\to p}[\transp{L}^N((D_q^Ba)(q, p, q^\prime)u(q^\prime))](\rho, \eta)| \nonumber\\
&\leq C_A\jbracket{p}^{|A|}\sum_{B\leq A} \|\transp{L}^N((D_q^B a)(q, p, q^\prime)u(q^\prime))\|_{L^1(dq^\prime)} \nonumber \\
&\leq C_A \jbracket{p}^{|A|-N} 
\sum_{\substack{B\leq A \\ |B_1|+|B_2|\leq N}} 
\| \partial_q^B \partial_{q^\prime}^{B_1} a(q, p, q^\prime)\partial_{q^\prime}^{B_2}u(q^\prime)\|_{L^1(dq^\prime)}. \label{est2.1.}
\end{align}

Now we assume that $q=(r, \theta)\in K$ and $\mathrm{supp}(u)\subset K^\prime$ for compact sets $K, K^\prime$. By the above lemma \ref{lemm_bs_comp}, 
\begin{align*}
&\| \partial_q^B \partial_{q^\prime}^{B_1}a(q, p, q^\prime)\partial_{q^\prime}^{B_2}u(q^\prime)\|_{L^1(dq^\prime)} \\ 
&\leq C_{K^\prime}\|\partial_q^B \partial_{q^\prime}^{B_1} a(q, p, q^\prime)\|_{L^\infty(K\times\mathbb{R}^n\times K^\prime)}\|\partial_q^{B_2}u\|_{L^\infty} \\
&\leq C_{A NKK^\prime t}|a|_{\lambda: |A|+N, m, 0, t}\jbracket{p}^m\|\partial_q^{B_2}u\|_{L^\infty}. 
\end{align*}
Hence 
\[
(\ref{est2.1.})\leq C_{ANKK^\prime t} \jbracket{p}^{|A|+m-N} |a|_{\lambda: |A|+N, m, 0, t}
\sum_{|B_2|\leq N} \| \partial_q^{B_2}u\|_{L^\infty}. 
\]
If we take $N$ so large that $|A|+m-N<-n$ (we take the smallest one), then $\jbracket{p}^{|A|+m-N}$ is integrable in $d\rho d\eta$. Hence so is $|F(q, p)|$, and we obtain the differentiability of $\int F(q, p)\, d\rho d\eta$. 

The continuity of $\Op(a)$ follows from 
\begin{align*}
&\sup_{q\in K}\left|\partial_q^A \left(\int F(q, p)\, dp\right)\right| 
\leq \sup_{q\in K} \int |\partial_q^AF(q, p)|\, dp  \\
&\leq C_{ANKK^\prime t} \int \jbracket{p}^{|A|+m-N}\,dp \,|a|_{\lambda: |A|+N, m, 0, t}
\sum_{|B_2|\leq N} \| \partial_q^{B_2}u\|_{L^\infty} \\
&\leq C_{ANKK^\prime t} |a|_{\lambda: |A|+N, m, 0, t}\sum_{|B_2|\leq N} \| \partial_q^{B_2}u\|_{L^\infty}. 
\qedhere 
\end{align*} 
\end{proof}

Now we state the $L^2$ boundedness of pseudodifferential operators with bounded bisymbols.

\begin{theo}\label{thm_L2_bdd_bi}
Let $\lambda: \mathbb{R}\to (0, \infty)$ be a function which satisfies the condition  (\ref{assump_lambda}). Then there exists a positive constant $C>0$ and an integer $N\geq 0$ such that 
\[
\|\Op^g(a)u\|_{L^2(g^{1/2}dq)\to L^2(g^{1/2}dq)}\leq C|a|_{\lambda: M, 0, 0, t}\|u\|_{L^2(g^{1/2}dq)}
\]
for all $t\in [0, 1]$, $a\in BS^0_0(\lambda; t)$ and $u\in C_c^\infty(\mathbb{R}^n)$. 
\end{theo}

\section{Proof of $L^2$ boundedness}\label{sec_L2bdd}

\subsection{Main argument}

Let $\{\psi_j=\psi(\cdot-j)\}_{j\in \mathbb{Z}}$ be a partition of unity of $\mathbb{R}$ where $\psi\in C_c^\infty((-1, 1))$,  $\psi\geq 0$. The multiplication operator by the function $\psi_j(r)$ is denoted as $\psi_j$ too: 
\[
(\psi_ju)(r, \theta):=\psi_j(r)u(r, \theta). 
\]

The important step for proving $L^2$ boundedness of $\Op(a)$ is an estimate of the $L^2$ operator norm of $\psi_j\Op(a)\psi_k$. 

\begin{prop}\label{prop_disjoint_union}
For any $N\geq 0$, there exists a constant $C=C_N>0$ and an integer $M=M_N>0$ such that 
\[
\|\psi_j\Op(a)\psi_ku\|_{L^2(dq)}\leq C|a|_{\lambda: M, 0, 0, t}\jbracket{j-k}^{-N}\|u\|_{L^2(dq)}
\]
holds for all $t\in [0, 1]$, $a\in BS(\lambda; t)$, $j$, $k\in\mathbb{Z}$ and $u\in C_c^\infty(\mathbb{R}^n)$. 
\end{prop}

In order to prove this proposition, we employ a kind of scaling arguments. The scaling operator $U^t_{jk}$ dependent on $j$, $k\in \mathbb{Z}$ and $t\in [0, 1]$ is defined as
\[
U^t_{jk}f(r, \theta):=(\lambda^t_{jk})^{-\frac{n-1}{2}}f(r, (\lambda^t_{jk})^{-1}\theta), 
\]
where
\[
\lambda^t_{jk}:=\lambda(tj+(1-t)k). 
\]
This is a unitary operator on $L^2(dq)$. Then the conjugation of $\psi_j\Op(a)\psi_k: C_c^\infty(\mathbb{R}^n)\to C^\infty(\mathbb{R}^n)$ by $U^t_{jk}$ is 
\begin{align*}
&U^t_{jk}\psi_j\Op(a)\psi_k(U^t_{jk})^{-1}u(q) \\
&=\frac{1}{(2\pi)^n}\int_{\mathbb{R}^n}dp \int_{\mathbb{R}^n}dq^\prime \, a\left(r, (\lambda^t_{jk})^{-1}\theta, p, q^\prime\right)\psi_j(r)\psi_k(r^\prime)e^{i\rho(r-r^\prime)+i\eta\cdot ((\lambda^t_{jk})^{-1}\theta-\theta^\prime)} \\
&\quad\times u(r^\prime, \lambda^t_{jk}\theta^\prime) \\
&=\frac{1}{(2\pi)^n}\int_{\mathbb{R}^n}d\tilde p \int_{\mathbb{R}^n}d\tilde{q}^\prime \, a\left(r, (\lambda^t_{jk})^{-1}\theta, \rho, \lambda^t_{jk}\tilde\eta, r^\prime, (\lambda^t_{jk})^{-1}\tilde{\theta}^\prime\right) \\
&\quad \times \psi_j(r)\psi_k(r^\prime)e^{i\tilde p\cdot (q-\tilde{q}^\prime)}u(\tilde{q}^\prime) \\
&=\Op(a^t_{jk})u(q). 
\end{align*}
Here we put 
\[
a^t_{jk}(q, p, q^\prime):=a\left(r, (\lambda^t_{jk})^{-1}\theta, \rho, \lambda^t_{jk}\eta, r^\prime, (\lambda^t_{jk})^{-1}\theta^\prime\right)\psi_j(r)\psi_k(r^\prime)
\]
and changed the variables 
\[
\tilde{q}^\prime=(r, \tilde\theta)=(r, \lambda^t_{jk}\theta), \quad \tilde p=(\rho, \tilde\eta)=(\rho, (\lambda^t_{jk})^{-1}\eta). 
\]
Thus if $\Op(a^t_{jk})$ is bounded on $L^2(dq)$, then $\psi_j\Op(a)\psi_k$ is also bounded on $L^2(dq)$ and they have the same operator norm on $L^2(dq)$. 

\begin{lemm}\label{lemm_scaling_bisymbols}
If $a\in BS^m_\sigma(\lambda; t)$, then $a^t_{jk}\in BS^m_\sigma(1; t)$ for all $j$, $k\in\mathbb{Z}$. 

Moreover, for each integer $M\geq 0$, there exists a constant $C_M>0$ such that 
\[
|a^t_{jk}|_{1: M, m, \sigma, t}\leq C_M |a|_{\lambda: M, m, \sigma, t}
\]
holds for all $t\in [0, 1]$, $a\in BS^m_\sigma(\lambda; t)$ and $j$, $k\in \mathbb{Z}$. 
\end{lemm}

\begin{proof}
If $|A|+|B|+|A^\prime|\leq M$, then, on the support of $a^t_{jk}$, 
\begin{align*}
&|\partial_q^A\partial_p^B\partial_{q^\prime}^{A^\prime}a^t_{jk}(q,p, q^\prime)| \\ 
&\leq C_M(\lambda^t_{jk})^{-|\alpha|+|\beta|-|\alpha^\prime|} \\
&\quad\times
\left|
\sum_{\substack{0\leq\tilde \alpha_0\leq \alpha_0 \\ 0\leq \tilde \alpha_0^\prime\leq \alpha_0^\prime}}
(\partial_r^{\tilde \alpha_0}\partial_\theta^\alpha\partial_p^B\partial_{r^\prime}^{\tilde \alpha_0^\prime}\partial_{\theta^\prime}^{\alpha^\prime} a)(r, (\lambda^t_{jk})^{-1}\theta, \rho, \lambda^t_{jk}\eta, r^\prime, (\lambda^t_{jk})^{-1}\theta^\prime)\right. \\
&\left.\quad\times\partial_r^{\alpha_0-\tilde \alpha_0}\psi_j(r)\partial_{r^\prime}^{\alpha_0^\prime-\tilde \alpha_0^\prime}\psi_k(r^\prime)\right| \\
&\leq C_M|a|_{\lambda: M, m, \sigma, t}
\jbracket{\rho\oplus (\lambda^t_{jk})^{-1}\lambda^t_{jk}\eta}^{m-\sigma|B|}
=C_M|a|_{\lambda: M, m, \sigma, t}
\jbracket{p}^{m-\sigma|B|}. \qedhere
\end{align*}
\end{proof}

Note that the class of bisymbols $BS^m_\sigma(1; t)$ is independent of $t\in [0, 1]$ and is just the usual bisymbol classes
\begin{align*}
&BS^m_\sigma(1) \\
&:=\{\,a\in C^\infty(\mathbb{R}^{3n})\mid |\partial_q^A \partial_p^B \partial_{q^\prime}^{A^\prime}a(q, p, q^\prime)|\leq C_{ABA^\prime}\jbracket{p}^{m-\sigma|B|}, \, \exists C_{ABA^\prime}>0\,\}. 
\end{align*}
The seminorms $|\cdot|_{1; M, m, \sigma, t}$ in $BS^m_\sigma(1)=BS^m_\sigma(1; t)$ is simply denoted as $|\cdot|_{M, m, \sigma}$. 
We employ some facts for the pseudodifferential operators with bisymbols which is stated below. 

\begin{prop}\label{prop_usual_bisymbols}
\begin{enumerate}
\item There exists a constant $C>0$ and $M\geq 0$ such that 
\[
\|\Op(a)u\|_{L^2(dq)}\leq C|a|_{M, 0, 0}\|u\|_{L^2(dq)}
\]
holds for all $a\in BS^0_0(1)$ and $u\in C_c^\infty(\mathbb{R}^n)$. 
\item For all $N\geq 0$, there exists a constant $C>0$ and an integer $M\geq 0$ such that 
\begin{align*}
&\|\chi_1\Op(a)\chi_2u\|_{L^2(dq)} \\
&\leq C\mathrm{dist}(\mathrm{supp}(\chi_1), \mathrm{supp}(\chi_2))^{-N}|\chi_1|_M |\chi_2|_M\sum_{|B|=N}|\partial_p^B a|_{M, 0, 0} \|u\|_{L^2(dq)}
\end{align*}
holds for all $a\in BS^0_0$, $u\in C_c^\infty(\mathbb{R}^n)$ and $\chi_1$, $\chi_2\in \mathcal{B}$ with 
\[
\mathrm{dist}(\mathrm{supp}(\chi_1), \mathrm{supp}(\chi_2))>0. 
\]
Here 
\[
\mathcal{B}:=\{\, \chi\in C^\infty(\mathbb{R}^n)\mid \partial_q^A\chi\in L^\infty(\mathbb{R}^n)\text{ for all } A\in \mathbb{Z}_{\geq 0}^n\,\}
\]
and 
\[
|\chi|_M:=\sum_{|A|\leq M} \|\partial_q^A\chi\|_{L^\infty(\mathbb{R}^n)}. 
\]
\end{enumerate}
\end{prop}

We prove Proposition \ref{prop_usual_bisymbols} in Appendix \ref{appendix_usual_psiDO}. What we do here is the proof of Proposition \ref{prop_disjoint_union} from Proposition \ref{prop_usual_bisymbols}. 

\begin{proof}[Proof of Proposition \ref{prop_disjoint_union}]
By the statement 1 of Proposition \ref{prop_usual_bisymbols} and Lemma \ref{lemm_scaling_bisymbols}, we obtain
\begin{align*}
&\|\psi_j\Op(a)\psi_k\|_{L^2(dq)\to L^2(dq)} \\
&=\|\Op(a^t_{jk})\|_{L^2(dq)\to L^2(dq)}\leq C|a^t_{jk}|_{M, 0, 0}\leq C|a|_{\lambda: M, m, \sigma, t}
\end{align*}
for some $C>0$ and $M\geq 0$ independent of $t\in [0,1]$ and $j$, $k\in \mathbb{Z}$. 

Assume that $|j-k|\geq 2$. Since $\mathrm{supp}(\psi)\subset (-1, 1)$, there exists a small $\delta>0$ such that $\mathrm{supp}(\psi)\subset (-1+\delta, 1-\delta)$. Take a smooth function $\tilde\psi: \mathbb{R}\to [0, \infty)$ such that
\[
\mathrm{supp}(\tilde\psi)\subset \left(-1+\frac{\delta}{2}, 1-\frac{\delta}{2}\right), \quad \tilde\psi=1 \text{ on } \mathrm{supp}(\psi) 
\]
and put $\tilde\psi_j:=\tilde\psi(\cdot-j)$. 
Then 
\[
\mathrm{dist}(\mathrm{supp}(\tilde\psi_j), \mathrm{supp}(\tilde\psi_k))\geq |j-k|-2+\delta>0
\]
and $\Op(a^t_{jk})=\tilde\psi_j\Op(a^t_{jk})\tilde\psi_k$. Hence we can apply the statement 2 of Proposition \ref{prop_usual_bisymbols} and obtain
\begin{align*}
&\|\Op(a^t_{jk})u\|_{L^2(dq)}=\|\tilde\psi_j\Op(a^t_{jk})\tilde\psi_ku\|_{L^2(dq)} \\
&\leq 
C_N(|j-k|-2+\delta)^{-N}|\tilde\psi|_{M_N}^2\sum_{|B|=N}|\partial_p^Ba^t_{jk}|_{M_N, 0, 0}\|u\|_{L^2(dq)} \\
&\leq C_N\jbracket{j-k}^{-N}\sum_{|B|=N}|\partial_p^Ba^t_{jk}|_{M_N, 0, 0}\|u\|_{L^2(dq)} \\
&\leq C_N\jbracket{j-k}^{-N}|a^t_{jk}|_{M_N+N, 0, 0}\|u\|_{L^2(dq)}
\leq C_N\jbracket{j-k}^{-N}|a|_{\lambda: M_N, 0, 0}\|u\|_{L^2(dq)}
\end{align*}
for some $C_N>0$ and $M_N\geq 0$, and for all $u\in C_c^\infty(\mathbb{R}^n)$. 
\end{proof}

We return to the proof of Theorem \ref{thm_L2_bdd_bi}. We need two more lemmas. 
\begin{lemm}\label{lem_formal_adjoint}
If $a\in BS^m_\sigma(\lambda; t)$, then $a^\dagger(q, p, q^\prime):=\overline{a(q^\prime, p, q)}\in BS^m_\sigma(\lambda; 1-t)$ and 
\[
(\psi_j\Op(a)\psi_k)^*|_{C_c^\infty(\mathbb{R}^n)}=\psi_k\Op(a^\dagger)\psi_j|_{C_c^\infty(\mathbb{R}^n)} 
\]
for $a\in BS^0_0(\lambda)$. 
\end{lemm}

\begin{proof}
The condition for $a^\dagger\in BS^m_\sigma(\lambda; 1-t)$ is equivalent to that for $a\in BS^m_\sigma(\lambda; t)$ by definition. 
Proposition \ref{prop_smoothness} justifies the argument in proving that $\Op(a^\dagger)$ is a formal adjoint of $\Op(a)$. Hence 
\[
\jbracket{\psi_j\Op(a)\psi_ku, v}=\jbracket{u, \psi_k\Op(a^\dagger)\psi_jv}
\]
for all $u$, $v\in C_c^\infty(\mathbb{R}^n)$. Hence
\[
\psi_k\Op(a^\dagger)\psi_jv=(\psi_j\Op(a)\psi_k)^*v
\]
for all $v\in C_c^\infty(\mathbb{R}^n)$. 
\end{proof}

\begin{lemm}[Cotlar-Stein lemma]
Let $\{A_\alpha: \mathcal{H}_1\to \mathcal{H}_2\}_{\alpha\in \Lambda}$ be a countable family of bounded operators between two Hilbert spaces $\mathcal{H}_1$ and $\mathcal{H}_2$. If 
\[
\sup_{\alpha\in \Lambda}\sum_{\beta\in \Lambda}\|A_\alpha^*A_\beta\|^\frac{1}{2}\leq M \text{ and } 
\sup_{\alpha\in \Lambda}\sum_{\beta\in \Lambda}\|A_\alpha A_\beta^*\|^\frac{1}{2}\leq M, 
\]
then 
\[
A:=\sum_{\alpha\in \Lambda}A_\alpha
\]
converges in a strong operator topology and $\|A\|\leq M$. 
\end{lemm}

We can find the proof of the Cotlar-Stein lemma in \cite{Martinez02}, \cite{Zworski12} for instance. 

\begin{proof}[Proof of Theorem \ref{thm_L2_bdd_bi}]
It is enough to show that
\[
\|\Op(a)u\|_{L^2(dq)\to L^2(dq)}\leq C|a|_{\lambda: M, 0, 0, t}\|u\|_{L^2(dq)}. 
\]
Put $A_{jk}:=\psi_j\Op(a)\psi_k$. We first want to prove
\[
A_{jk}A_{lm}^*=1_{\{|k-m|\leq 1\}}A_{jk}A_{lm}^*, \, A_{jk}^*A_{lm}=1_{\{|j-l|\leq 1\}}A_{jk}^*A_{lm},  
\]
which are formally obvious. 
For a technical reason, we take a smooth function $\chi\in C_c^\infty(\mathbb{R}^n)$ such that $\chi(q)=1$ near $q=0$ and define $\chi_\varepsilon:=\chi(\varepsilon\cdot)$. 
Let $u\in C_c^\infty(\mathbb{R}^n)$. Then 
\[
\chi_\varepsilon A_{lm}^*u=\chi_\varepsilon \psi_m\Op(a^\dagger)\psi_lu\in C_c^\infty(\mathbb{R}^n)
\]
by Lemma \ref{lem_formal_adjoint} and Proposition \ref{prop_smoothness}. Thus
\[
A_{jk}\chi_\varepsilon A_{lm}^*u=\psi_j\Op(a)\psi_k\chi_\varepsilon\psi_m\Op(a^\dagger)\psi_lu
=1_{\{|k-m|\leq 1\}}A_{jk}\chi_\varepsilon A_{lm}^*u. 
\]
We take a limit $\varepsilon \to 0$ in a $L^2(dq)$-norm topology and obtain
\[
A_{jk}A_{lm}^*u=1_{\{|k-m|\leq 1\}}A_{jk}A_{lm}^*u 
\]
for all $u\in C_c^\infty(\mathbb{R}^n)$. Here we used the $L^2(dq)$ boundedness of $A_{jk}$. 

The other equality $A_{jk}^*A_{lm}=1_{\{|j-l|\leq 1\}}A_{jk}^*A_{lm}$ is proved similarly. Since $\chi_\varepsilon A_{lm}u\in C_c^\infty(\mathbb{R}^n)$ for $u\in C_c^\infty(\mathbb{R}^n)$ by Proposition \ref{prop_smoothness}, we obtain 
\[
A_{jk}^*\chi_\varepsilon A_{lm}u=\psi_k\Op(a^\dagger)\psi_j\chi_\varepsilon \psi_m \Op(a)\psi_lu
=1_{\{|j-m|\leq 1\}}A_{jk}^*\chi_\varepsilon A_{lm}u
\]
by applying Lemma \ref{lem_formal_adjoint}. We only have to take a limit $\varepsilon\to 0$ in an $L^2(dq)$-norm topology. 

Now we estimate
\[
\sum_{l, m\in \mathbb{Z}}\|A_{jk}A_{lm}^*\|^\frac{1}{2}\leq C|a|_{\lambda: M, 0, 0, t}\sum_{\substack{l, m\in \mathbb{Z} \\ |k-m|\leq 1}}\jbracket{j-k}^{-\frac{3}{2}}\jbracket{l-m}^{-\frac{3}{2}}\leq C|a|_{\lambda: M, 0, 0, t}
\]
and 
\[
\sum_{l, m\in \mathbb{Z}}\|A_{jk}^*A_{lm}\|^\frac{1}{2}\leq C|a|_{\lambda: M, 0, 0, t}\sum_{\substack{l, m\in \mathbb{Z} \\ |j-l|\leq 1}}\jbracket{j-k}^{-\frac{3}{2}}\jbracket{l-m}^{-\frac{3}{2}}\leq C|a|_{\lambda: M, 0, 0, t}
\]
for some $C>0$ and $M\geq 0$ independent of $t\in [0, 1]$, $a\in BS^0_0(\lambda; t)$ and $j$, $k\in\mathbb{Z}$ by Proposition \ref{prop_disjoint_union} setting $N=3$. Hence, by the Cotlar-Stein lemma, we obtain
\[
\left\|\sum_{j, k\in \mathbb{Z}}A_{jk}\right\|_{L^2(dq)\to L^2(dq)}\leq C|a|_{\lambda: M, 0, 0, t}. \qedhere
\]
\end{proof}

\subsection{The semiclassical case}

We record the semiclassical estimate of semiclassical pseudodifferential operators here, although we do not use it in the argument of the parametrices of resolvents. 

\begin{defi}
Let $\hbar>0$ be a small parameter. For $a\in S^m_\sigma(\lambda)$, we define
\[
\Op^{g, t}_\hbar(a):=\Op^{g, t}(a(q, \hbar p)), \, \Op^t_\hbar(a):=\Op^t(a(q, \hbar p)). 
\]
For $a\in BS^m_\sigma(\lambda)$, we define 
\[
\Op^g_\hbar (a):=\Op^g(a(q, \hbar p, q^\prime)), \, \Op_\hbar (a):=\Op(a(q, \hbar p, q^\prime)). 
\]
\end{defi}

Noting that 
\[
\sum_{|B|=N}|\partial_p^B(a(q, \hbar p, q^\prime))|_{M, 0, 0}=\hbar^N\sum_{|B|=N}|\partial_p^Ba|_{M, 0, 0}, 
\]
we have the semiclassical estimate of $\|\psi_j\Op_\hbar(a)\psi_k\|_{L^2(dq)\to L^2(dq)}$ for $|j-k|\geq 2$. 

\begin{prop}
For all $N\geq 0$, there exist constants $C>0$ and $M\geq 0$ such that 
\[
\|\psi_j\Op_\hbar(a)\psi_k\|_{L^2(dq)\to L^2(dq)}\leq C\hbar^N\jbracket{j-k}^{-N}|a|_{\lambda: M, 0, 0, t}
\]
for all $t\in [0, 1]$, $a\in BS(\lambda; t)$, $|j-k|\geq 2$ and $\hbar\in (0, 1]$. 
\end{prop}

\begin{prop}\label{prop_scaling_curved}
There exist constants $C>0$ and $M\geq 0$ such that 
\[
\|\psi_j\Op_\hbar(a)\psi_ku\|_{L^2(dq)}\leq C(\|a\|_{L^\infty(\mathbb{R}^{3n})}+\hbar^\frac{1}{2}|a|_{\lambda: M, 0, 0, t})\|u\|_{L^2(dq)}
\]
holds for all $t\in [0, 1]$, $a\in BS(\lambda; t)$, $j$, $k\in\mathbb{Z}$, $\hbar\in (0, 1]$ and $u\in C_c^\infty(\mathbb{R}^n)$. 
\end{prop}

The basic fact is the proposition below, which is proved by the standard semiclassical scaling argument. (For the detail, see Appendix \ref{appendix_scaling}.)
\begin{prop}\label{prop_scaling_flat}
There exists a constant $C>0$ and $M\geq 0$ such that 
\[
\|\Op_\hbar(a)\|_{L^2(dq)\to L^2(dq)}\leq C\sum_{|A|+|B|+|A^\prime|\leq M} \hbar^{\frac{1}{2}(|A|+|B|+|A^\prime|)} \|\partial_q^A\partial_p^B\partial_{q^\prime}^{A^\prime}a\|_{L^\infty(\mathbb{R}^{3n})}
\]
holds for all $a\in BS^0_0(1)$ and $\hbar\in (0, 1]$. 
\end{prop}

\begin{proof}[Proof of Proposition \ref{prop_scaling_curved}]
First, by Proposition \ref{prop_scaling_flat}, we obtain 
\begin{align*}
&\|\psi_j\Op_\hbar(a)\psi_k\|_{L^2(dq)\to L^2(dq)}=\|\Op_\hbar(a^t_{jk})\|_{L^2(dq)\to L^2(dq)} \\
&\leq C\sum_{|A|+|B|+|A^\prime|\leq M} \hbar^{\frac{1}{2}(|A|+|B|+|A^\prime|)} \|\partial_q^A\partial_p^B\partial_{q^\prime}^{A^\prime}a^t_{jk}\|_{L^\infty(\mathbb{R}^{3n})}. 
\end{align*}
The right hand side is smaller than
\[
C(\|a\|_{L^\infty(\mathbb{R}^{3n})}+\hbar^\frac{1}{2}|a|_{\lambda: M, 0, 0, t})
\]
by Lemma \ref{lemm_scaling_bisymbols}. This completes the proof. 
\end{proof}

\begin{theo}
There exist constants $C>0$ and $M\geq 0$ such that 
\[
\|\Op_\hbar(a)\|_{L^2(dq)\to L^2(dq)}\leq C(\|a\|_{L^\infty(\mathbb{R}^{3n})}+\hbar^\frac{1}{2}|a|_{\lambda: M, 0, 0, t})
\]
holds for all $t\in [0, 1]$, $a\in BS(\lambda; t)$, $j$, $k\in\mathbb{Z}$ and $\hbar\in (0, 1]$. 
\end{theo}


\section{Construction of the parametrix of resolvents}\label{sec_parametrix}

\subsection{Composition with differential operators}

First we define differential operators acting on functions in terms of polar coordinates. We introduce coefficient function classes 
\[
\mathcal{B}(\Omega, \lambda):=\{\, a\in C^\infty(\Omega)\mid (\lambda^{-1}\partial_\theta)^\alpha\partial_r^{\alpha_0}a\in L^\infty(\Omega), \, \forall (\alpha_0, \alpha)\in \mathbb{Z}_{\geq 0}^n\,\}. 
\]
for an open subset $\Omega \subset \mathbb{R}^n$. If $\Omega=\mathbb{R}^n$, we denote $\mathcal{B}(\mathbb{R}^n, \lambda)$ as $\mathcal{B}(\lambda)$. Although we can equip these sets with a topological structure by seminorms $\| (\lambda^{-1}\partial_\theta)^\alpha\partial_r^{\alpha_0}a\|_{L^\infty}$, we do not use it in this paper. 

\begin{defi}\label{def_differential_operator}
Let $\Omega \subset \mathbb{R}^n$ be an open subset. A linear operator $P: C_c^\infty(\Omega)\to C_c^\infty(\Omega)$ is a differential operator belonging to the set $\mathrm{Diff}^m(\Omega, \lambda)$ if $P$ is the form 
\[
P=\sum_{|\Gamma|\leq m} a_\Gamma(q)(\lambda^{-1}D_\theta)^\gamma D_r^{\gamma_0}, \quad a_\Gamma \in \mathcal{B}(\Omega, \lambda). 
\]
We denote $\mathrm{Diff}^m(\mathbb{R}^n, \lambda)$ as $\mathrm{Diff}^m(\lambda)$. 
The principal symbol is defined as 
\[
\sigma(P)(q, p):=\sum_{|\Gamma|= m}a_\Gamma(q)\rho^{\gamma_0}\eta^\gamma 
\]
for $(q, p)\in \Omega\times\mathbb{R}^n$. 
\end{defi}

We add a suitable condition to $\lambda: \mathbb{R}\to (0, \infty)$. 

\begin{assump}\label{cond_w2}
$\lambda: \mathbb{R}\to (0, \infty)$ is a smooth function satisfying
\[
\partial_r^{j+1}\log \lambda(r)\in L^\infty(\mathbb{R}) \quad \text{for all } j\geq 0. 
\]
\end{assump}

It is easy to check that Assumption \ref{cond_w2} implies the condition (\ref{assump_lambda}). In following, we always assume that $\lambda: \mathbb{R}\to (0, \infty)$ satisfies Assumption \ref{cond_w2}. 

\begin{rema*}
If $\lambda$ satisfies Assumption \ref{cond_w2}, then $\mathrm{Diff}(\Omega, \lambda):=\cup_{m\geq 0} \mathrm{Diff}^m(\Omega, \lambda)$ forms an algebra with respect to the product defined as the composition of derivatives, with the property
\[
\mathrm{Diff}^{m_1}(\Omega, \lambda)\cdot \mathrm{Diff}^{m_2}(\Omega, \lambda)\subset \mathrm{Diff}^{m_1+m_2}(\Omega, \lambda) \quad \forall m_1, \, m_2\geq 0.
\]
\end{rema*}


We consider the composition of differential operators and pseudodifferential operators $\Op^1(b)$. In following we denote $\Op^1(b)$ as $\Op(b)$ for simplicity. 

\begin{prop}\label{prop_composition_do_pdo}
For $b\in S^m_\sigma(\lambda)$ and $u\in C_c^\infty(\mathbb{R}^n)$, we obtain 
\[
(\lambda^{-1}D_\theta)^\gamma D_r^{\gamma_0}\Op(b)u=\Op((\lambda^{-1} \eta+\lambda^{-1} D_\theta)^\gamma (\rho+D_r)^{\gamma_0}b)u. 
\]
\end{prop}

\begin{proof}
The proof of Proposition \ref{prop_smoothness} shows that we can change the order of integration and differentiation. 
\end{proof}

We can calculate a composition of differential operators and pseudodifferential operators by Proposition \ref{prop_composition_do_pdo}. Note that the composed symbol $(\lambda^{-1} \eta+\lambda^{-1} D_\theta)^\gamma (\rho+D_r)^{\gamma_0}b$ is equal to 
\[
\rho^{\gamma_0}(\lambda^{-1}\eta)^\gamma b
+\sum_{\substack{B\leq \Gamma \\ |B|\geq 1}} \rho^{\gamma_0-\beta_0}(\lambda^{-1}\eta)^{\gamma-\beta}(\lambda^{-1}D_\theta)^\beta D_r^{\beta_0} b
\]
and belongs to $S^{m+|\Gamma|}_\sigma(\lambda)$ if $b\in S^m_\sigma(\lambda)$. 
Hence if $P=\sum_{|\Gamma|\leq m}a_\Gamma (q)(\lambda^{-1}D_\theta)^\gamma D_r^{\gamma_0}$, then the symbol of $P\mathrm{Op}(b)$ is $L_0(P)b+L_1(P)b+L_2(P)b+\cdots+L_m(P)b$, where
\[
L_k(P)b:=\sum_{\substack{B\leq \Gamma \\ |B|=k \\ |\Gamma|\leq m}} a_\Gamma(q)\rho^{\gamma_0-\beta_0}(\lambda^{-1}\eta)^{\gamma-\beta}(\lambda^{-1}D_\theta)^\beta D_r^{\beta_0} b. 
\]
In particular, $L_0(P)$ is a multiplication operator
\[
L_0(P)=\sum_{|\Gamma|\leq m}a_\Gamma(q)\rho^{\gamma_0}(\lambda^{-1}\eta)^\gamma. 
\]
We put
\[
\widetilde{L}_0(P):=L_0(P)-\sigma(P). 
\]

\subsection{Local calculation of parametrix of resolvent}

For a complex number $z\in\mathbb{C}$, we define 
\[
\delta(z, \sigma(P)):=\inf_{q, p}|z-\sigma(P)(q, p)|. 
\]

If we take $z$ such that $\delta(z, \sigma(P))>0$, then $(z-\sigma(P))^{-1}=1/(z-\sigma(P))$ is well-defined as both function and multiplication operator. We add an ellipticity condition to differential operators in order to restrict the behavior of $(z-\sigma(P))^{-1}$ in a high energy area.  

\begin{defi}\label{def_elliptic}
Let $\Omega \subset \mathbb{R}^n$ be an open subset. A differential operator $P\in \mathrm{Diff}^m(\Omega, \lambda)$ is \textit{elliptic} if for all $z\in\mathbb{C}$ with $\delta(z, \sigma(P))>0$ there exists a constant $C>0$ such that 
\[
C^{-1}\jbracket{\rho\oplus \lambda(r)^{-1}\eta}^m\leq |z-\sigma(P)(q, p)|\leq C\jbracket{\rho\oplus \lambda(r)^{-1}\eta}^m
\]
holds for all $(q, p)\in \Omega\times \mathbb{R}^n$. 
\end{defi}

\begin{prop}\label{prop_parametrix}
Fix a function $\lambda: \mathbb{R}\to (0, \infty)$ satisfying Assumption \ref{cond_w2}. Let $P\in \mathrm{Diff}^m(\Omega, \lambda)$ be the form $P=\sum_{|\Gamma|\leq m}a_\Gamma (q)(\lambda^{-1}D_{\theta})^\gamma D_r^{\gamma_0}$. Suppose that $P$ is elliptic and $z\in\mathbb{C}$ satisfies $\delta(z, \sigma(P))>0$. We also take a function (of $q$) $\chi\in \mathcal{B}(\lambda)$ such that $\mathrm{supp}(\chi)\subset \Omega$. 

If we define symbols $b_0^N(z)=b_0(z)+b_1(z)+\cdots +b_N(z)$ by
\begin{align}
&b_0(z)(q, p):=\chi(q)(z-\sigma(P)(q, p))^{-1}, \nonumber \\
&b_j(z):=(z-\sigma(P))^{-1}\left( \widetilde{L}_0(P)b_{j-1}(z)+\sum_{0\leq k<j} L_{j-k}(P)b_k(z)\right) \quad \text{for } j\geq 1, \label{eq_bjs}
\end{align} 
then $b_j(z)\in S^{-m-j}_1(\lambda)$, $\mathrm{supp}(b_j(z))\subset \mathrm{supp}(\chi)\times \mathbb{R}^n$ and we have
\[
(z-P)\mathrm{Op}(b_0^N(z))u=\chi u+\Op(e_{N+1}(z))u
\]
for all $u\in C_c^\infty(\mathbb{R}^n)$ and some symbol $e_{N+1}(z, h)\in S^{-N-1}_1(\lambda)$ supported in $\mathrm{supp}(\chi)\times \mathbb{R}^n$ with estimate
\begin{equation}
|e_{N+1}(z)|_{\lambda: M, 0, 1, 1}\leq C\delta(z, \sigma(P))^{-\frac{N+1}{m}}\sum_{l=0}^{M+N+1}\sup_{q, p}\left(\frac{\jbracket{\rho\oplus \lambda^{-1}\eta}^m}{|z-\sigma(P)|}\right)^{l+1-\frac{N+1}{m}}  \label{ineq_remainder}
\end{equation}
for some $C=C_M>0$. 
\end{prop}

\begin{proof}
A direct computation shows that the symbol of $(z-P)\Op_h(b_0^N(z))$ is
\begin{align*}
&\sum_{k=0}^N \left((z-\sigma(P))b_k(z)-1_{\{k\geq 1\}}\widetilde{L}_0(P)b_{k-1}(z)-\sum_{j=0}^{k-1} L_{k-j}(P)b_j(z)\right) \\
&+e_{N+1}(z)
\end{align*}
if we put
\[
e_{N+1}(z):=-\widetilde{L}_0(P)b_N(z)-\sum_{\substack{j\leq m, k\leq N \\ j+k>N}} L_{j}(P)b_k(z). 
\]
We only have to solve the algebraic equations
\[
(z-\sigma(P))b_0(z)=\chi, \, (z-\sigma(P))b_j(z)-\widetilde{L}_0(P)b_{j-1}(z)-\sum_{0\leq k<j} L_{j-k}(P)b_k(z)=0
\]
for $1\leq j\leq N$ to obtain formulas for symbols. $\mathrm{supp}(b_j(z))\subset \mathrm{supp}(\chi)\times \mathbb{R}^n$ is obvious from these equations. 

We estimate the symbol $e_{N+1}(z)$. We introduce a shorthand notation 
\[
\partial_\lambda^\mathfrak{a}:=(\lambda^{-1}\partial_\theta)^\alpha (\lambda\partial_\eta)^\beta \partial_r^{\alpha_0}\partial_\rho^{\beta_0}
\]
for $\mathfrak{a}=(A, B)=(\alpha_0, \alpha, \beta_0, \beta)$.  We also abbreviate $L_k(P)$ to $L_k$.  First we record that the derivative of $(z-\sigma(P))^{-k}$ for $|B|\geq 1$ is the form 
\[
\partial_\lambda^\mathfrak{a}(z-\sigma(P))^{-k}=\sum_{l=0}^{|\mathfrak{a}|}(z-\sigma(P))^{-l-k}c_{kl, \mathfrak{a}}, 
\]
where $c_{kl, \mathfrak{a}}$ is 
\begin{align*}
c_{kl,\mathfrak{a}}(q, p)
&=\sum_{\substack{\mathfrak{a}=\mathfrak{a}_1+\cdots+\mathfrak{a}_l \\ |\mathfrak{a}_1|, \ldots, |\mathfrak{a}_l|\geq 1}} \partial_\lambda^{\mathfrak{a}_1}\sigma(P) \cdots\partial_\lambda^{\mathfrak{a}_l}\sigma(P) \\
&\in \sum_{\substack{\mathfrak{a}=\mathfrak{a}_1+\cdots+\mathfrak{a}_l \\ |\mathfrak{a}_1|, \ldots, |\mathfrak{a}_l|\geq 1}} 
S^{m-|B_1|}_1\cdots S^{m-|B_l|}_1 
\subset S^{ml-|B|}_1(\lambda). 
\end{align*}

We estimate the symbols $b_j(z)$ and $e_{N+1}(z)$ applying this fact. The procedure of estimate will be divided into 5 steps. 


\noindent \textbf{Step 1.} 
The derivative $\partial_\lambda^\mathfrak{a}b_0(z)$ is equal to 
\[
\sum_{\mathfrak{b}\leq \mathfrak{a}}
\begin{pmatrix}
\mathfrak{a} \\
\mathfrak{b}
\end{pmatrix}
\partial_\lambda^{\mathfrak{a}-\mathfrak{b}}\chi
\sum_{l=0}^{|\mathfrak{b}|}(z-\sigma(P))^{-l-1}c_{1l, \mathfrak{b}}
=\sum_{l=0}^{|\mathfrak{a}|}(z-\sigma(P))^{-l-1}c^0_{l, \mathfrak{a}}. 
\]
Here we put 
\[
c^0_{l, \mathfrak{a}}:=
\sum_{\substack{\mathfrak{b}\leq \mathfrak{a} \\ |\mathfrak{b}|\geq l}}
\begin{pmatrix}
\mathfrak{a} \\
\mathfrak{b}
\end{pmatrix}
c_{1l, \mathfrak{b}}
\partial_\lambda^{\mathfrak{a}-\mathfrak{b}}\chi 
\in S^{ml-l}_1(\lambda). 
\]
Thus 
\[
|\partial_\lambda^\mathfrak{a}b_0(z)|
\leq C_\mathfrak{a}|z-\sigma(P)|^{-1}\sum_{l=0}^{|\mathfrak{a}|}\frac{|c_{1l, \mathfrak{a}}|}{|z-\sigma(P)|^l}\leq C_\mathfrak{a}\delta(z, L_0)^{-1}\jbracket{\rho\oplus \lambda^{-1}\eta}^{-|B|}. 
\]
Hence
\[
|b_0(z)|_{\lambda: M, 0, 1, 1}=\sum_{|\mathfrak{a}|\leq M}\sup\jbracket{\rho\oplus \lambda^{-1}\eta}^{-|B|}|\partial_\lambda^\mathfrak{a} b_0(z)|
\leq C_M\delta(z, L_0)^{-1}. 
\]

\noindent \textbf{Step 2.} 
Next we show that $\partial_\lambda^\mathfrak{a}b_j(z)$ is the form
\begin{equation}
\partial_\lambda^\mathfrak{a}b_j(z)=\sum_{l=0}^{|B|+2j-1}(z-\sigma(P))^{-l-2}c^j_{l, \mathfrak{a}}, \label{derivative_bj}
\end{equation}
with symbols $c^j_{l, \mathfrak{a}}\in S^{m(l+1)-j-|B|}_1(\lambda)$ independent of $z$ by induction in $j\geq 1$. It is convenient to introduce more shorthand notations for representing differential operators $L_k$ as
\[
L_k=\sum_{\substack{\Gamma^\prime\leq \Gamma  \\ |\Gamma^\prime|=k \\ |\Gamma|\leq m}}a_\Gamma(q)p_\lambda^{\Gamma-\Gamma^\prime}D_{q\lambda}^{\Gamma^\prime}, \,\, 
p_\lambda^{\Gamma-\Gamma^\prime}=\rho^{\gamma_0-\gamma_0^\prime}(\lambda^{-1}\eta)^{\gamma-\gamma^\prime}, \,\, 
D_{q\lambda}^{\Gamma^\prime}=(\lambda^{-1}D_\theta)^{\gamma^\prime} D_r^{\gamma_0^\prime}. 
\]

We first calculate $b_1(z)$. 
\begin{align*}
&b_1(z)=(z-\sigma(P))^{-1}(\widetilde{L}_0(P)b_0(z)+L_1b_0(z)) \\
&=(z-\sigma(P))^{-1}\left(\widetilde{L}_0(P)b_0(z)+\sum_{\substack{\Gamma^\prime\leq \Gamma \\ |\Gamma^\prime|=1 \\ |\Gamma|\leq m}} a_\Gamma(q)p_\lambda^{\Gamma-\Gamma^\prime}D_{q\lambda}^{\Gamma^\prime} b_0(z)\right) \\ 
&=\sum_{l=0}^1(z-\sigma(P))^{-l-2}c^1_{l0}. 
\end{align*}
Here we defined 
\[
c^1_{l0}(q, p):=
\delta_{0l}\sum_{|\Gamma|< m}a_\Gamma(q)\rho^{\gamma_0}(\lambda^{-1}\eta)^\gamma
+\sum_{\substack{\Gamma^\prime\leq \Gamma \\ |\Gamma^\prime|=1 \\ |\Gamma|\leq m}} i^{-|\Gamma^\prime|}a_\Gamma(q)p_\lambda^{\Gamma-\Gamma^\prime} c^0_{l, \Gamma^\prime},
\]
\[
c^0_{l, \Gamma^\prime}:=c_{1l, (\Gamma^\prime, 0)}
\]
for $\Gamma^\prime\in \mathbb{Z}_{\geq 0}^n$ and $l\in \{0, 1\}$. 
It turns out that $c^1_{l0}\in S^{m(l+1)-1}_1(\lambda)$, since 
\[
a_\Gamma p_\lambda^{\Gamma-\Gamma^\prime}c^0_{l, \Gamma^\prime}\in S^{|\Gamma|-|\Gamma^\prime|+ml}_1\subset S^{(m-l^\prime)-(1-l^\prime)+ml}_1=S^{m(l+1)-1}_1
\]
if $|\Gamma^\prime|+l^\prime=1$ and $|\Gamma|+l^\prime\leq m$. The derivatives are calculated as 
\begin{align*}
\partial_\lambda^\mathfrak{a}b_1(z)
&=\sum_{l=0}^1 \sum_{\mathfrak{a}_1+\mathfrak{a}_2=\mathfrak{a}} C_{\mathfrak{a}_1\mathfrak{a}_2}\partial_\lambda^{\mathfrak{a}_1}(z-\sigma(P))^{-l-2}\partial_\lambda^{\mathfrak{a}_2}c^1_{l, 0} \\
&=\sum_{l=0}^1 \sum_{\mathfrak{a}_1+\mathfrak{a}_2=\mathfrak{a}} C_{\mathfrak{a}_1\mathfrak{a}_2}\sum_{l^\prime=0}^{|\mathfrak{a}_1|}(z-\sigma(P))^{-l-2-l^\prime}c_{l+1, l^\prime, \mathfrak{a}_1}\partial_\lambda^{\mathfrak{a}_2}c^1_{l, 0} \\
&=\sum_{l=0}^{|\mathfrak{a}|+1}(z-\sigma(P))^{-l-2}c^1_{l, \mathfrak{a}}. 
\end{align*}
Here we put
\[
c^1_{l\mathfrak{a}}(q, p):=\sum_{\substack{\mathfrak{a}_1+\mathfrak{a}_2=\mathfrak{a} \\ 0\leq l^\prime\leq |\mathfrak{a}_1| \\ l^\pprime\geq 0 \\ l^\prime+l^\pprime=l}}C_{\mathfrak{a}_1\mathfrak{a}_2}c_{l^\pprime+1, l^\prime, \mathfrak{a}_1}(q, p)\partial_\lambda^{\mathfrak{a}_2}c^1_{l^\pprime, 0}(q, p) 
\]
for $|\mathfrak{a}|\geq 1$. 
$c^1_{l, \mathfrak{a}}$ belongs to $S^{m(l+1)-1-|B|}_1$ since 
\[
c_{l^\pprime+1, l^\prime, \mathfrak{a}_1}\partial_\lambda^{\mathfrak{a}_2}c^0_{l^\pprime, \mathfrak{a}_2}
\in S^{ml^\prime-|B_1|+m(l^\pprime+1)-1-|B_2|}_1=S^{m(l+1)-1-|B|}_1
\]
if $l^\prime+l^\pprime=l$ and $B_1+B_2=B$. 

\noindent \textbf{Step 3.} 
Assume that this holds for all $j^\prime\leq j-1$. For simplicity, we abuse the notation $c^k_{l, A}:=c^k_{l, (A, 0)}$ for $A\in \mathbb{Z}_{\geq 0}^n$. 
We calculate
\begin{align*}
L_{j-k}b_k(z)
&=\sum_{\substack{\Gamma^\prime\leq \Gamma \\ |\Gamma^\prime|=j-k \\ |\Gamma|\leq m}} a_\Gamma(q)p_\lambda^{\Gamma-\Gamma^\prime} \sum_{l=0}^{|\Gamma^\prime|+2k-1}(z-\sigma(P))^{-l-2}c^k_{l, \Gamma^\prime} \\
&=\sum_{l=1}^{j+k}(z-\sigma(P))^{-l-1}\sum_{\substack{\Gamma^\prime\leq \Gamma\geq 0 \\ |\Gamma^\prime|=j-k \\ |\Gamma|\leq m}} a_\Gamma(q)p_\lambda^{\Gamma-\Gamma^\prime} c^k_{l-1, \Gamma^\prime}
\end{align*}
for $1\leq k\leq j-1$ and 
\begin{align*}
L_jb_0(z)
&=\sum_{\substack{\Gamma^\prime\leq \Gamma \\ |\Gamma^\prime|=j \\ |\Gamma|\leq m}} a_\Gamma (q)p_\lambda^{\Gamma-\Gamma^\prime} \sum_{l=0}^{|\Gamma^\prime|}(z-\sigma(P))^{-l-1}c^0_{l, \Gamma^\prime} \\
&=\sum_{l=0}^j(z-\sigma(P))^{-l-1}\sum_{\substack{\Gamma^\prime\leq \Gamma \\ |\Gamma^\prime|=j \\ |\Gamma|\leq m}} a_\Gamma(q)p_\lambda^{\Gamma-\Gamma^\prime} c^0_{l, \Gamma^\prime}
\end{align*}
for $k=0$. We sum these up over $0\leq k\leq j-1$ and a trivial identity $\widetilde{L}_0(P)=\widetilde{L}_0(P)$. The result is 
\[
b_j(z)=
\sum_{l=0}^{2j-1} (z-\sigma(P))^{-l-2} c^j_{l, 0}, 
\]
where 
\[
c^j_{0, 0}(q, p):=
\sum_{\substack{\Gamma^\prime\leq \Gamma \\ |\Gamma^\prime|=j \\ |\Gamma|\leq m \\ }} a_\Gamma (q)p_\lambda^{\Gamma-\Gamma^\prime} c^0_{0, \Gamma^\prime}
\]
and 
\begin{align*}
&c^j_{l, 0}(q, p):= \\
&1_{\{l\leq 2j-2\}}\sum_{|\Gamma|< m}a_\Gamma(q)\rho^{\gamma_0}(\lambda^{-1}\eta)^\gamma c^{j-1}_{l-1, 0}
+
1_{\{l\leq j\}}\sum_{\substack{\Gamma^\prime\leq \Gamma \\ |\Gamma^\prime|=j \\ |\Gamma|\leq m}} a_\Gamma (q)p_\lambda^{\Gamma-\Gamma^\prime} c^0_{l, \Gamma^\prime} \\
&+\sum_{\substack{1\leq k\leq j-1 \\ j+k\geq l}}\sum_{\substack{\Gamma^\prime\leq \Gamma   \\ |\Gamma^\prime|=j-k \\ |\Gamma|\leq m}} a_\Gamma (q)p_\lambda^{\Gamma-\Gamma^\prime} c^k_{l-1, \Gamma^\prime} 
\end{align*}
for $l\geq 1$. These symbols satisfy $c^j_{l, 0}\in S^{m(l+1)-j}_1(\lambda)$ for all $l=0, \ldots, 2j-1$. 

\noindent \textbf{Step 4.} 
Finally we calculate $\partial_\lambda^\mathfrak{a}b_j(z)$, $\mathfrak{a}\in \mathbb{Z}_{\geq 0}^n\times \{ B\}$. The method is same as that in the calculation of $\partial_\lambda^\mathfrak{a}b_1(z)$. The result is 
\begin{align*}
&\partial_\lambda^\mathfrak{a}b_j(z)=\sum_{l=0}^{|\mathfrak{a}|+2j-1}(z-\sigma(P))^{-l-2}c^j_{l, \mathfrak{a}}, \\
&c^j_{l, \mathfrak{a}}(q, p):=\sum_{\substack{\mathfrak{a}_1+\mathfrak{a}_2=\mathfrak{a} \\ 0\leq l^\prime\leq |\mathfrak{a}_1| \\ l^\pprime\geq 0 \\ l^\prime+l^\pprime=l}}C_{\mathfrak{a}_1\mathfrak{a}_2}c_{l^\pprime+2, l^\prime, \mathfrak{a}_1}(q, p)c^j_{l^\pprime, 0}(q, p)\in S^{m(l+1)-j-|B|}_1(\lambda). 
\end{align*}
Now the proof of (\ref{derivative_bj}) is completed. 

\noindent \textbf{Step 5.} 
Recall the definition of $e_{N+1}(z)$ 
\[
e_{N+1}(z)=-\widetilde{L}_0(P)b_N(z)-\sum_{\substack{j\leq m, k\leq N \\ j+k>N}} L_jb_k(z). 
\]

Since
\begin{align*}
L_jb_k(z)
&=\sum_{l=1}^{j+k}(z-\sigma(P))^{-l-1}\sum_{\substack{\Gamma^\prime\leq \Gamma \\ |\Gamma^\prime|=j \\ |\Gamma|\leq m}} a_\Gamma (q)p_\lambda^{\Gamma-\Gamma^\prime} c^k_{l-1, \Gamma^\prime} \\
&\in \sum_{l=1}^{j+k}(z-\sigma(P))^{-l-1}S^{m(l+1)-j-k}_1
\end{align*}
for $k\geq 1$ and 
\begin{align*}
L_jb_0(z)
&=\sum_{l=0}^j(z-\sigma(P))^{-l-1}\sum_{\substack{\Gamma^\prime\leq \Gamma \\ |\Gamma^\prime|=j \\ |\Gamma|\leq m}} a_\Gamma (q)p_\lambda^{\Gamma-\Gamma^\prime} c^0_{l, \Gamma^\prime} \\
&\in \sum_{l=0}^j(z-\sigma(P))^{-l-1} S^{m(l+1)-j}_1, 
\end{align*}
we obtain 
\begin{align*}
e_{N+1}(z) \in 
&\sum_{l=0}^{2N-1}(z-\sigma(P))^{-l-2}S^{m(l+1)-N+m-1}_1 \\
&+\sum_{\substack{j\leq m, k\leq N \\ j+k>N}} \left(1_{\{k\geq 1\}}\sum_{l=1}^{j+k}(z-\sigma(P))^{-l-1}S^{m(l+1)-j-k}_1\right. \\
&\qquad\qquad\quad\left.+\delta_{k0}\sum_{l=0}^j(z-\sigma(P))^{-l-1} S^{m(l+1)-j}_1\right) \\
\subset 
&\sum_{l=0}^{N^\prime}(z-\sigma(P))^{-l-1}S^{m(l+1)-N-1}_1. 
\end{align*}
Here $N^\prime:=\max\{ 2N-1, 1\}$. 
Note that we can more precisely denote $e_{N+1}(z)=\sum_l (z-\sigma(P))^{-l-1}f_l$ with $f_l\in S^{m(l+1)-N-1}_1$ independent of $z$. This implies that $e_{N+1}(z)\in S^{-N-1}_1(\lambda)$. The derivative of $e_{N+1}(z)$ is
\begin{align*}
\partial_\lambda^\mathfrak{a}e_{N+1}(z) 
&\in \sum_{l=0}^{N+1} \sum_{\mathfrak{a}_1+\mathfrak{a}_2=\mathfrak{a}}C_{\mathfrak{a}_1\mathfrak{a}_2}\sum_{l^\prime=0}^{|\mathfrak{a}_1|}(z-\sigma(P))^{-l^\prime-l-1}c_{l+1, l^\prime,  \mathfrak{a}_1}S^{m(l+1)-N-1-|B_2|}_1 \\
&\subset  \sum_{l=0}^{N+1} \sum_{\mathfrak{a}_1+\mathfrak{a}_2=\mathfrak{a}}\sum_{l^\prime=0}^{|\mathfrak{a}_1|}(z-\sigma(P))^{-l^\prime-l-1}S^{ml^\prime-|B_1|}_1S^{m(l+1)-N-1-|B_2|}_1 \\
&\subset \sum_{l=0}^{N+1+|\mathfrak{a}|}(z-\sigma(P))^{-l-1}S^{m(l+1)-N-1-|B|}_1
\end{align*}
We estimate $|\partial_\lambda^\mathfrak{a}e_{N+1}(z, h)|$ as
\begin{align*}
&\jbracket{\rho\oplus \lambda^{-1}\eta}^{|B|}|\partial_\lambda^\mathfrak{a}e_{N+1}(z, h)| \\
&\leq C_\mathfrak{a}\sum_{l=0}^{N+1+|\mathfrak{a}|} |z-\sigma(P)|^{-l-1}\jbracket{\rho\oplus \lambda^{-1}\eta}^{m(l+1)-N-1} \\
&\leq C_\mathfrak{a} \delta(z, \sigma(P))^{-\frac{N+1}{m}}\sum_{l=0}^{N+1+|\mathfrak{a}|}\sup_{q, p}\left(\frac{\jbracket{\rho\oplus \lambda^{-1}\eta}^m}{|z-\sigma(P)|}\right)^{l+1-\frac{N+1}{m}}. 
\end{align*}
This is the result which we want. 
\end{proof}

\subsection{Weighted quantization}

Let $g: \mathbb{R}^n\to (0, \infty)$ be a smooth function in $\mathcal{B}(\lambda)$ which has a positive infimum. Note that $g^a$ also belongs to $\mathcal{B}(\lambda)$ for all $a\in \mathbb{R}$. We introduce differential operators
\[
\mathcal{D}_r:=g^{-\frac{1}{4}}D_rg^\frac{1}{4}, \, \mathcal{D}_{\theta^j}:=g^{-\frac{1}{4}}D_{\theta^j}g^\frac{1}{4}. 
\]
Then $\mathcal{D}_r$ and $\lambda^{-1}\mathcal{D}_{\theta^j}$ are differential operators in $\mathrm{Diff}^1(\lambda)$. For general $P\in\mathrm{Diff}^m(\Omega, \lambda)$, $g^aPg^{-a}\in \mathrm{Diff}^m(\Omega, \lambda)$ and the (modified) principal symbol is invariant: $\sigma(P)=\sigma(g^aPg^{-a})$. 

\begin{coro}
Fix a positive function $\lambda(r)$ satisfying Assumption \ref{cond_w2}. Let $P\in \mathrm{Diff}^m(\lambda)$ be the form $P=\sum_{|\Gamma|\leq m}a_\Gamma (q)(\lambda^{-1}\mathcal{D}_{\theta})^\gamma \mathcal{D}_r^{\gamma_0}$. Suppose that $P$ is elliptic in $\Omega$ and $z\in\mathbb{C}$ satisfies $\delta(z, \sigma(P))>0$. We also take a function (of $q$) $\chi\in \mathcal{B}(\lambda)$ such that $\mathrm{supp}(\chi)\subset \Omega$. 

If we define symbols $b_N(z)=b_0(z)+b_1(z)+\cdots +b_N(z)$ by the same formulas as in (\ref{eq_bjs}), 
then $b_j(z)\in S^{-m-j}_1(\lambda)$, $\mathrm{supp}(b_j(z))\subset \mathrm{supp}(\chi)\times \mathbb{R}^n$ and we have
\[
(z-P)\Op^g(b_N(z))u=\chi u+\Op^g(e_{N+1}(z))u
\]
for all $u\in C_c^\infty(\mathbb{R}^n)$ and some symbol $e_{N+1}(z, h)\in S^{-N-1}_1(\lambda)$ supported in $\mathrm{supp}(\chi)\times \mathbb{R}^n$ with the estimate as in (\ref{ineq_remainder}). 
\end{coro}

\section{Resolvents of differential operators on manifolds with ends}\label{sec_ess_selfadj}

\subsection{Differential operators on manifolds with ends}

Let $\mathcal{M}$ be a manifold with ends. We take a finite atlas $\{\Psi_\iota: \mathcal{U}_\iota\to \mathcal{V}_\iota\}_{\iota\in I}$ as in introduction. We equip $\mathcal{M}$ with a positive function $\lambda: \mathcal{M}\to (0, \infty)$ which appears in the Assumption \ref{assump_metric_M}. Recall that for each $\iota\in I_\infty$, $\lambda_\iota:=(\Psi_\iota^{-1})^*\lambda$ is a function depending only on $r\in (1, \infty)$. 

\begin{defi}\label{def_elliptic_M}
An operator $\mathcal{P}: C_c^\infty(\mathcal{M})\to C_c^\infty(\mathcal{M})$ is an \textit{differential operator on} $\mathcal{M}$ \textit{of degree at most} $m$ if 
\begin{itemize}
\item $\mathrm{supp}(Pu)\subset \mathrm{supp}(u)$ for all $u\in C^\infty(\mathcal{M})$, 
\item for all $\iota\in I_\infty$, $(\Psi_\iota^{-1})^*\mathcal{P}\Psi_\iota^*\in \mathrm{Diff}^m(\mathcal{V}_\iota, \lambda_\iota)$, 
\item for all $\iota \in I_\mathcal{K}$, $(\Psi_\iota^{-1})^*\mathcal{P}\Psi_\iota^*\in\mathrm{Diff}^m(\mathcal{V}_\iota, 1)$. 
\end{itemize}
The set of all differential operators on $\mathcal{M}$ of degree at most $m$ is denoted as $\mathrm{Diff}^m(\mathcal{M}, \lambda)$. 
The principal symbol $\sigma(\mathcal{P}): T^*\mathcal{M}\to \mathbb{C}$ of $\mathcal{P}\in \mathrm{Diff}^m(\mathcal{M}, \lambda)$ is defined as 
\[
\sigma(\mathcal{P})(x, \xi):=\sigma(\Psi_\iota^*\mathcal{P}(\Psi_\iota^{-1})^*)(\tilde\Psi_\iota(x, \xi)). 
\]
Here $\iota \in I$ is an index such that $x\in \mathcal{U}_\iota$. $\sigma(\mathcal{P})$ is independent of the choice of such $\iota\in I$. $\tilde\Psi_\iota: T^*\mathcal{U}_\iota\to \mathcal{V}_\iota\times \mathbb{R}^n$ is the canonical coordinates associated with $\Psi_\iota: \mathcal{U}_\iota\to\mathcal{V}_\iota$. 
\end{defi}

We define the ellipticity of differential operators on $\mathcal{M}$ in coordinate-free terms. 

\begin{defi}
A differential operator $\mathcal{P}\in \mathrm{Diff}^m(\mathcal{M}, \lambda)$ is \textit{elliptic} if, for all $z\in\mathbb{C}$ with $\mathrm{dist}(z, \sigma(\mathcal{P})(T^*\mathcal{M}))>0$, there exists a constant $C>0$ such that the inequality
\[
C^{-1}(1+|\xi|_{g^{-1}})^m\leq |z-\sigma(\mathcal{P})(x, \xi)| \leq C(1+|\xi|_{g^{-1}})^m
\]
holds for all $(x, \xi)\in T^*\mathcal{M}$.  
\end{defi}

We can easily characterize the ellipticity of differential operators on $\mathcal{M}$ by coordinate-dependent terms. 

\begin{prop}
$\mathcal{P}\in \mathrm{Diff}^m(\mathcal{M})$ is elliptic if and only if $(\Psi_\iota^{-1})^*\mathcal{P}\Psi_\iota^*\in \mathrm{Diff}^m(\mathcal{V}_\iota, \lambda_\iota)$ is elliptic for all $\iota\in I_\infty$ and $(\Psi_\iota^{-1})^*\mathcal{P}\Psi_\iota^*\in \mathrm{Diff}^m(\mathcal{V}_\iota, 1)$ is elliptic for all $\iota\in I_\mathcal{K}$. 
\end{prop}


\subsection{Proof of main theorem}

\begin{proof}[Proof of Theorem \ref{thm_parametrix}]
Let $\{\kappa_\iota\}_{\iota\in I}$ be a partition of unity subordinate to $\{\mathcal{U}_\iota\}_{\iota\in I}$ such that 
\begin{itemize}
\item $\sum_{\iota \in I}\kappa_\iota(x)^2=1$ for all $x\in \mathcal{M}$, 
\item for each $\iota\in I_\infty$, $\kappa_\iota(r, \sigma)=\kappa_\iota^S(\sigma)$ if $(r, \sigma)\in [2, +\infty)\times \mathcal{U}_\iota^\prime$. Here $\{\kappa_\iota^S\}_{\iota\in I_\infty}$ is a partition of unity subordinate to a finite atlas $\{\mathcal{U}_\iota^\prime\}_{\iota\in I_\infty}$ on $S$. 
\end{itemize}
Note that each $\kappa_\iota$ above is not necessarily compactly supported. 
$\tilde\kappa_\iota$ belongs to $\mathcal{B}(\lambda_\iota)$ since $\lambda_\iota$ satisfies $\liminf_{r\to \infty}\lambda_\iota(r)>0$ in Assumption \ref{assump_metric_M}. 

Put $P_\iota:=(\Psi_\iota^{-1})^*\mathcal{P}\Psi_\iota^*\in \mathrm{Diff}^m(\lambda_\iota)$ and $g_\iota(q):=\det(g_{\mu\nu}^\iota(q))_{\mu, \nu}$. We construct symbols $b^\iota_j(z)\in S^{-m-j}_1(\lambda_\iota)$ ($j=0, 1, \ldots , N$) such that $\mathrm{supp}(b^\iota_j(z))\subset \mathcal{V}_\iota \times \mathbb{R}^n$ and 
\[
(z-P_\iota)\Op^{g_\iota}(b^\iota_N(z))u=\kappa_\iota u+\Op^{g_\iota}(e^\iota_{N+1}(z))u
\]
for all $u\in C_c^\infty(\mathcal{V}_\iota)$. The symbol $e^\iota_{N+1}(z, h)\in S^{-N-1}_1(\lambda_\iota)$ has estimate
\[
|e^\iota_{N+1}(z)|_{\lambda: M, 0, 1, 1}\leq C\delta(z, \sigma(P_\iota))^{-\frac{N+1}{m}}\sum_{l=0}^{M+N+1}\sup_{q, p}\left(\frac{\jbracket{\rho\oplus \lambda_\iota^{-1}\eta}^m}{|z-\sigma(P_\iota)|}\right)^{l+1-\frac{N+1}{m}} 
\] 
for some $C=C_M>0$. $b_j(z)$ and $e_{N+1}(z)$ are supported in $\mathrm{supp}(\tilde\kappa_\iota)\times \mathbb{R}^n\subset \mathcal{V}_\iota\times \mathbb{R}^n$. 

We set 
\[
\mathcal{Q}^\iota_N(z):=\Psi_\iota^*\Op^{g_\iota}(b^\iota_N(z))\tilde\kappa_\iota(\Psi_\iota^{-1})^*: 
C_c^\infty(\mathcal{M})\to C^\infty(\mathcal{M}) 
\]
and 
\[
\mathcal{R}^\iota_N(z):=\Psi_\iota^*\Op^{g_\iota}(e^\iota_{N+1}(z))\tilde\kappa_\iota(\Psi_\iota^{-1})^*: 
C_c^\infty(\mathcal{M})\to C^\infty(\mathcal{M}).  
\]
Then 
\begin{equation}
(z-\mathcal{P})\mathcal{Q}_N^\iota(z)u=\tilde\kappa_\iota^2 u+\mathcal{R}_{N+1}^\iota(z)u. \label{eq_part_parametrix}
\end{equation}
The $L^2$ boundedness theorem \ref{thm_L2_bdd_simplest} gives us the estimate of remainder terms 
\begin{align*}
&\|\mathcal{R}^\iota_{N+1}(z)\|_{L^2(\mathcal{M}, g)\to L^2(\mathcal{M}, g)} \\
&\leq C\delta(z, \sigma(P_\iota))^{-\frac{N+1}{m}}\sum_{l=0}^{N+M}\sup_{(q, p)\in \mathcal{V}_\iota\times\mathbb{R}^n}\left(\frac{\jbracket{\rho\oplus \lambda_\iota^{-1}\eta}^m}{|z-\sigma(P_\iota)|}\right)^{l+1-\frac{N+1}{m}} \\
&\leq C\delta(z, \sigma(\mathcal{P}))^{-\frac{N+1}{m}}\sum_{l=0}^{N+M}\sup_{(x, \xi)\in T^*\mathcal{M}}\left(\frac{(1+|\xi|_{g^{-1}})^m}{|z-\sigma(\mathcal{P})|}\right)^{l+1-\frac{N+1}{m}}
\end{align*}
for some integer $M\in \mathbb{N}$. 

For $\iota\in I_\mathcal{K}$, we construct $L^2(\mathcal{M}, g)$ bounded operators $\mathcal{Q}^\iota_N(z)$ and $\mathcal{R}^\iota_N(z)$ satisfying the above equation (\ref{eq_part_parametrix}) by usual pseudodifferential operators. Thus summing (\ref{eq_part_parametrix}) over $\iota\in I$, we obtain 
\[
(z-\mathcal{P})\sum_{\iota\in I}\mathcal{Q}^\iota_N(z)u=u+\sum_{\iota\in I}\mathcal{R}^\iota_{N+1}(z)u. 
\]
Hence $\mathcal{Q}_N(z):=\sum_{\iota\in I}\mathcal{Q}^\iota_N(z): L^2(\mathcal{M}, g)\to L^2(\mathcal{M}, g)$ is an approximate resolvent and $\mathcal{R}_{N+1}(z):=\sum_{\iota\in I}\mathcal{R}^\iota_{N+1}(z): L^2(\mathcal{M}, g)\to L^2(\mathcal{M}, g)$ is the remainder term. 
\end{proof}

\subsection{Essential self-adjointness of elliptic differential operators}

As an application, we investigate fundamental properties of elliptic differential operators on manifolds with ends. In this section we focus on the essential self-adjointness of symmetric and elliptic differential operators. 

\begin{theo}
Suppose that $\mathcal{P}\in \mathrm{Diff}^m(\mathcal{M}, \lambda)$ is elliptic and symmetric in $L^2(\mathcal{M}, g)$. Then $\mathcal{P}$ is essentially self-adjoint in $L^2(\mathcal{M}, g)$. 
\end{theo}

\begin{proof}
We prove that $(\pm iy-\mathcal{P})C_c^\infty(\mathcal{M})$ is dense in $L^2(\mathcal{M}, g)$ for $y\in\mathbb{R}$ such that $|y|\gg 1$. Construct a parametrix $\mathcal{Q}(z):=\mathcal{Q}_0(z)$ as in Theorem \ref{thm_parametrix}. Then 
\[
(iy-\mathcal{P})\mathcal{Q}(iy)u=u+\mathcal{R}(iy)u
\]
for all $u\in C_c^\infty(\mathcal{M})$. The remainder term $\mathcal{R}(iy)=\mathcal{R}_1(iy)$ has the estimate 
\[
\|\mathcal{R}(iy)\|_{L^2(\mathcal{M}, g)\to L^2(\mathcal{M}, g)}\leq C|y|^{-\frac{1}{m}}
\]
for all $|y|\geq 1$. This estimate is proved by the fact that the reality of principal symbol $\sigma(\mathcal{P})$ implies $|iy-\sigma(\mathcal{P})|\geq |i-\sigma(\mathcal{P})|$ and hence 
\[
\sup_{\substack{(x, \xi)\in \\ T^*\mathcal{M}}}\left(\frac{(1+|\xi|_{g^{-1}})^m}{|iy-\sigma(\mathcal{P})|}\right)^{l+1-\frac{1}{m}}\leq C
\]
for some $C>0$ independent of $y$. We take $|y|$ so large that 
\[
\|\mathcal{R}(iy)\|_{L^2(\mathcal{M}, g)\to L^2(\mathcal{M}, g)}<1. 
\]
Then $1+\mathcal{R}(iy)$ is invertible in $L^2(\mathcal{M}, g)$. 

Now we take an arbitrary $u\in L^2(\mathcal{M}, g)$ and $\varepsilon>0$. Since $1+\mathcal{R}(iy)$ is invertible and $C_c^\infty(\mathcal{M})$ is dense in $L^2(\mathcal{M}, g)$, we can take $v\in C_c^\infty(\mathcal{M})$ such that 
\[
\|u-(1+\mathcal{R}(iy))v\|_{L^2(\mathcal{M}, g)}< \frac{\varepsilon}{2}. 
\]

Next we take $\chi\in C_c^\infty(\mathbb{R})$ with $\chi(r)=1$ if $|r|<2$. For each $\iota\in I_\infty$ and $\delta>0$, we define 
\[
\mathcal{Q}^\iota_\delta(iy):=\Psi_\iota^*(\chi(\delta r)\Op^{g_\iota}(b^\iota_0(iy)))(\Psi_\iota^{-1})^*\kappa_\iota. 
\]
Then 
\begin{align*}
&\mathcal{P}\mathcal{Q}^\iota_\delta(iy)v-\mathcal{P}\mathcal{Q}^\iota(iy)v \\
&=\Psi_\iota^*([P_\iota, \chi(\delta r)]\Op^{g_\iota}(b^\iota_0(iy))-(1-\chi(\delta r))\Op^{g_\iota}(b^\iota_0(iy)))(\Psi_\iota^{-1})^*\kappa_\iota v. 
\end{align*}
The second term in the right hand side is easy to estimate as 
\[
\|\Psi_\iota^*((1-\chi(\delta r))\Op^{g_\iota}(b^\iota_0(iy)))(\Psi_\iota^{-1})^*\kappa_\iota v\|_{L^2(\mathcal{M}, g)}=o(1) \quad (\delta\to 0). 
\]
Thus 
\begin{align*}
&\|\mathcal{P}\mathcal{Q}^\iota_\delta(iy)v-\mathcal{P}\mathcal{Q}^\iota(iy)v\|_{L^2(\mathcal{M}, g)} \\
&\leq\|\Psi_\iota^*([P_\iota, \chi(\delta r)]\Op^{g_\iota}(b^\iota_0(iy)))(\Psi_\iota^{-1})^*\kappa_\iota v\|_{L^2(\mathcal{M}, g)}+o(1) \quad (\delta\to 0). 
\end{align*}

$[P_\iota, \chi(\delta r)]=\delta P^\prime_\iota(\delta)$ for some $P^\prime_\iota(\delta)=\sum_{|\Gamma|\leq m-1} a^\prime_{\Gamma, \delta}(q)(\lambda^{-1}\mathcal{D}_\theta)^\gamma \mathcal{D}_r^{\gamma_0} \in \mathrm{Diff}^{m-1}(\lambda_\iota)$ with coefficients $a^\prime_{\Gamma, \delta}\in \mathcal{B}(\lambda_\iota)$ uniformly bounded with respect to $\delta\ll 1$. Since $b^\iota_0(iy)\in S^{-m}_1(\lambda_\iota)$, the symbol of $P^\prime_\iota(\delta)\Op^{g_\iota}(b^\iota_0(iy))$ is in $S^{-1}_1(\lambda_\iota)\subset S^0_1(\lambda_\iota)$ and uniformly bounded with respect to $\delta\ll 1$. Thus 
\[
\|\Psi_\iota^*([P_\iota, \chi(\delta r)]\Op^{g_\iota}(b^\iota_0(iy)))(\Psi_\iota^{-1})^*\kappa_\iota v\|_{L^2(\mathcal{M}, g)}=O(\delta) \quad (\delta\to 0). 
\]
Hence if we define $\mathcal{Q}_\delta(iy):=\sum_{\iota\in I_\mathcal{K}}\mathcal{Q}^\iota(iy)+\sum_{\iota\in I_\infty} \mathcal{Q}^\iota_\delta(iy)$, then we have
\[
\|\mathcal{P}\mathcal{Q}_\delta(iy)v-\mathcal{P}\mathcal{Q}(iy)v\|_{L^2(\mathcal{M}, g)}\to 0 \quad (\delta\to 0). 
\]
This implies the existence of $\delta>0$ such that 
\[
\|\mathcal{P}\mathcal{Q}_\delta(iy)v-\mathcal{P}\mathcal{Q}(iy)v\|_{L^2(\mathcal{M}, g)}<\frac{\varepsilon}{2}. 
\]
We define $w:=\mathcal{Q}_\delta(iy)v\in C_c^\infty(\mathcal{M})$. Then $\mathcal{P}w$ is close to $u$ since
\[
\|\mathcal{P}w-u\|_{L^2(\mathcal{M}, g)}
\leq \|\mathcal{P}w-\mathcal{P}\mathcal{Q}(iy)v\|_{L^2(\mathcal{M}, g)}
+\|(1+\mathcal{R}(iy))v-u\|_{L^2(\mathcal{M}, g)}<\varepsilon. \qedhere
\]
\end{proof}

\begin{coro}
The Laplacian $-\triangle_g$ associated with $g$ is an elliptic differential operator in $\mathrm{Diff}^2(\mathcal{M}, \lambda)$ and essentially self-adjoint on $L^2(\mathcal{M}, g)$. 
\end{coro}

\appendix
\section*{Appendix}
\section{Facts on pseudodifferential operators with usual bisymbols}
\label{appendix_usual_psiDO}

\subsection{Proof of Proposition \ref{prop_usual_bisymbols}}
\begin{proof}[Proof of 1]
This is a well-known result. See, for instance, the textbook of Kumano-go \cite{Kumano-go81} (Combine Theorem 2.5 in Chapter 2 with Theorem 1.6 in Chapter 7). 
\end{proof}

\begin{proof}[Proof of 2]
Since $q\neq q^\prime$ on the support of $a(q, p, q^\prime)\chi_1(q)\chi_2(q^\prime)$, we can apply the integration by parts by a differential operator $-i(q-q^\prime)\cdot\partial_p/|q-q^\prime|^2$ and obtain
\[
\chi_1\Op(a)\chi_2u=\Op\left(\chi_1(q)\chi_2(q^\prime)\left(\frac{i(q-q^\prime)\cdot\partial_p}{|q-q^\prime|^2}\right)^Na(q, p, q^\prime)\right)u 
\]
for all $N\geq 0$ and $u\in C_c^\infty(\mathbb{R}^n)$. 

The derivative of the bisymbol is estimated as 
\begin{align*}
&\sum_{|A|+|B|+|A^\prime|\leq M}\left|\partial_q^A\partial_p^B\partial_{q^\prime}^{A^\prime}\left(\chi_1(q)\chi_2(q^\prime)\left(\frac{i(q-q^\prime)\cdot\partial_p}{|q-q^\prime|^2}\right)^Na(q, p, q^\prime)\right)\right| \\
&\leq C_{MN}|\chi_1|_M|\chi_2|_M\sum_{|B^\prime|=N}|\partial_p^{B^\prime}a|_{M, 0, 0}|q-q^\prime|^{-N}1_{\mathrm{supp}(\chi_1)}(q)1_{\mathrm{supp}(\chi_2)}(q^\prime) \\
&\leq C_{MN}\mathrm{dist}(\mathrm{supp}(\chi_1), \mathrm{supp}(\chi_2))^{-N}|\chi_1|_M |\chi_2|_M\sum_{|B^\prime|=N}|\partial_p^{B^\prime} a|_{M, 0, 0}  
\end{align*}
for all $M\geq 0$. 

Combining this estimate with the statement 1 in Proposition \ref{prop_usual_bisymbols}, we finish the proof. 
\end{proof}

\subsection{The semiclassical scaling}\label{appendix_scaling}
\begin{proof}[Proof of Proposition \ref{prop_scaling_flat}]
We define
\[
U_\hbar f(q):=\hbar^\frac{n}{4}f(\hbar^\frac{1}{2}q). 
\]
This is a unitary operator on $L^2(dq)$. For $a\in BS^0_0(1)$, 
\[
U_\hbar\Op_\hbar(a)U_h^{-1}=\Op(a_\hbar), 
\]
where
\[
a_\hbar(q, p, q^\prime):=a(\hbar^\frac{1}{2}q, \hbar^\frac{1}{2}p, \hbar^\frac{1}{2}q^\prime). 
\]
$a_\hbar$ is a bisymbol belonging to $BS^0_0(1)$ with 
\[
|a_\hbar|_{M, 0, 0, t}=\sum_{|A|+|B|+|A^\prime|\leq M} \hbar^{\frac{1}{2}(|A|+|B|+|A^\prime|)}
\|\partial_q^A\partial_p^B\partial_{q^\prime}^{A^\prime}a\|_{L^\infty(\mathbb{R}^{3n})}. 
\]
We apply the statement 1 in Proposition \ref{prop_usual_bisymbols} to $\Op(a_\hbar)$. 
\end{proof}

\subsection*{Acknowledgments}

I would like to thank Professor Shu Nakamura and Kenichi Ito for a lot of discussion and advice.


\end{document}